\documentclass{amsart}
\usepackage{amsmath, amsthm, amssymb, amsfonts}
\usepackage[normalem]{ulem}
\usepackage{hyperref}
\usepackage[all]{xy}
\usepackage{combelow}

\usepackage{verbatim}
\usepackage{caption}
\usepackage{mathtools}
\mathtoolsset{showonlyrefs}

\usepackage{tikz-cd}

\theoremstyle{plain}
\newtheorem{thm}{Theorem}[section]

\newtheorem{lem}[thm]{Lemma}
\newtheorem{lemma}[thm]{Lemma}
\newtheorem{prop}[thm]{Proposition}

\newtheorem{claim}[thm]{Claim}

\theoremstyle{definition}
\newtheorem{defn}[thm]{Definition}

\theoremstyle{remark}
\newtheorem{rmk}[thm]{Remark}

\newcommand{\BN}{{\mathbb{N}}}

\newcommand{\BQ}{{\mathbb{Q}}}

\newcommand{\BZ}{{\mathbb{Z}}}

\newcommand{\CO}{{\mathcal O}}

\newcommand{\CZ}{{\mathcal Z}}

\newcommand{\Fg}{{\mathfrak{g}}}
\newcommand{\Fh}{{\mathfrak{h}}}

\newcommand{\Fq}{{\mathfrak{q}}}

\newcommand{\Ft}{{\mathfrak{t}}}

\newcommand{\Fz}{{\mathfrak{z}}}

\newcommand{\fg}{{\mathfrak{g}}}
\newcommand{\fq}{{\mathfrak{q}}}
\newcommand{\fJ}{{\mathfrak{J}}}
\newcommand{\fG}{{\mathfrak{G}}}
\newcommand{\fL}{{\mathfrak{L}}}

\newcommand{\ch}{{\mathrm{ch}}}

\DeclareFontFamily{OT1}{rsfs}{}
\DeclareFontShape{OT1}{rsfs}{n}{it}{<-> rsfs10}{}
\DeclareMathAlphabet{\curly}{OT1}{rsfs}{n}{it}

\newcommand\Hom{\operatorname{Hom}}

\newcommand\Id{\operatorname{Id}}

\newcommand\Aut{\operatorname{Aut}}

\def\mult{{\text{mult}}}
\def\univ{{\text{univ}}}

\def\Tan{{\text{Tan}}}

\def\Hilb{{\textup{Hilb}}}
\def\eHilb{{\emph{Hilb}}}

\def\ch{{\text{ch}}}

\def\th{{\widetilde{h}}}
\def\gNS{{\Fg_{\mathrm{NS}}}}

\def\act{{\textup{act}}}

\def\secskip{\vspace{.5\linespacing plus.7\linespacing}}

\begin{document}
\title[Motivic decompositions for $\text{Hilb}_n(\text{K3})$]{Motivic decompositions for the Hilbert scheme of points of a K3 surface}
\date{\today}

\author{Andrei Negu\cb{t}}
\address{MIT, Department of Mathematics, Cambridge, MA, USA}
\address{Simion Stoilow Institute of Mathematics, Bucharest, Romania}
\email{andrei.negut@gmail.com}

\author{Georg Oberdieck}
\address{University of Bonn, Institut f\"ur Mathematik, Bonn, Germany}
\email{georgo@math.uni-bonn.de}

\author{Qizheng Yin}
\address{Peking University, BICMR, Beijing, China}
\email{qizheng@math.pku.edu.cn}

\begin{abstract}
We construct an explicit, multiplicative Chow--K\"unneth decomposition for the Hilbert scheme of points of a K3 surface. We further refine this decomposition with respect to the action of the Looijenga--Lunts--Verbitsky Lie algebra.
\end{abstract}

\maketitle

\section{Introduction}
In the present paper, we study the motivic aspects of the Looijenga--Lunts--Verbitsky (\cite{LL, V}, LLV for short) Lie algebra action on the Chow ring of the Hilbert scheme of points of a K3 surface. Using a special element of the LLV algebra and formulas of \cite{MN} by Maulik and the first author, we construct an explicit Chow--K\"unneth decomposition for the Hilbert scheme, prove its multiplicativity, and show that all divisor classes and Chern classes lie in the correct component of the decomposition. This confirms expectations of Beauville~\cite{Be} and Voisin \cite{Voi}. We also obtain a refined motivic decomposition for the Hilbert scheme by taking into account the LLV algebra action, and prove its multiplicativity.

Both results parallel the case of an abelian variety, which we shall briefly review.

\subsection{Abelian varieties}
Let $X$ be an abelian variety of dimension $g$. Recall the classical result of Deninger--Murre on the decomposition of the Chow motive $\Fh(X)$.
\begin{thm}[\cite{DM}]
There is a unique, multiplicative Chow--K\"unneth decomposition:
\begin{equation} \label{dmdec}
\mathfrak{h}(X) = \bigoplus_{i = 0}^{2g} \Fh^i(X)
\end{equation}
such that for all $N \in \BZ$, the multiplication $[N]: X \to X$ acts on $\Fh^i(X)$ by $[N]^*\! = N^i$.
\end{thm}

The decomposition \eqref{dmdec} specializes to the K\"unneth decomposition in cohomology (hence the name Chow--K\"unneth), and to the Beauville decomposition \cite{Be0} in Chow. The latter takes the form:
\begin{equation} \label{bdec}
A^*(X) = \bigoplus_{i, s} A^i(X)_s
\end{equation}
with:
\[A^i(X)_s = A^i(\Fh^{2i - s}(X)) = \{\alpha \in A^i(X) \,|\, [N]^*\alpha = N^{2i - s} \alpha \text{ for all } N \in \BZ\}.\]
The \emph{multiplicativity} of \eqref{dmdec} stands for the fact that the cup product:
\[\cup: \Fh(X) \otimes \Fh(X) \to \Fh(X)\]
respects the grading, in the sense that: 
\[\cup: \Fh^i(X) \otimes \Fh^j(X) \to \Fh^{i + j}(X)\]
for all $i, j \in \{0, ..., 2g\}$. This can be seen by simply comparing the actions of $[N]^*$. As a result, the bigrading in~\eqref{bdec} is multiplicative, \emph{i.e.}, compatible with the ring structure of $A^*(X)$.

The Beauville decomposition is expected to provide a multiplicative \emph{splitting} of the conjectural Bloch--Beilinson filtration on $A^*(X)$. A difficult conjecture of Beauville (and consequence of the Bloch--Beilinson conjecture) predicts the vanishing $A^*(X)_s = 0$ for $s < 0$ and the injectivity of the cycle class map:
\[\mathrm{cl}: A^*(X)_0 \to H^*(X).\]

Further, any symmetric ample class $\alpha \in A^1(X)_0$ induces an $\mathfrak{sl}_2$-triple $(e_\alpha, f_\alpha, h)$ acting on $A^*(X)$. A Lefschetz decomposition of $\Fh(X)$ with respect to the $\mathfrak{sl}_2$-action was obtained by K\"unnemann \cite{Ku}, refining \eqref{dmdec}. More generally, Moonen~\cite{Mon} constructed an action of the N\'eron--Severi part of the Looijenga--Lunts \cite{LL} Lie algebra~$\Fg_{\mathrm{NS}}$ on $A^*(X)$, which contains all possible $\mathfrak{sl}_2$-triples above (he actually considered the slightly larger Lie algebra $\mathfrak{sp}(X \times X^\vee)$; see \cite[Section 6]{Mon}). He then obtained a refined motivic decomposition with respect to the $\Fg_{\mathrm{NS}}$-action.

\begin{thm}[\cite{Mon}]
There is a unique decomposition:
\begin{equation} \label{modec}
\Fh(X) = \bigoplus_{\psi \in \mathrm{Irrep}(\Fg_{\mathrm{NS}})}\Fh_\psi(X)
\end{equation}
where $\psi$ runs through all isomorphism classes of finite-dimensional irreducible representations of $\Fg_{\mathrm{NS}}$, and $\Fh_\psi(X)$ is $\psi$-isotypic under $\gNS$.
\end{thm}

Here being \emph{$\psi$-isotypic} means that $\Fh_\psi(X)$ is stable under $\gNS$ and
that for any Chow motive~$M$, the $\gNS$-representation $\Hom(M, \Fh_\psi(X))$ is isomorphic to a direct sum of copies of $\psi$. 

Again \eqref{modec} specializes to refined decompositions in cohomology and in Chow.

\subsection{Chow--K\"unneth}
We switch to the case of the Hilbert scheme. Let $S$ be a projective K3 surface over an algebraically closed field of characteristic $0$, and let~$X = \Hilb_n(S)$ be the Hilbert scheme of $n$ points on~$S$.

In~\cite{Ob}, the second author lifted the action of the N\'eron--Severi part of the LLV algebra~$\Fg_{\mathrm{NS}}$ from cohomology to Chow. In particular, there is an explicit grading operator:
\[h \in A^{2n}(X \times X)\]
which appears in every $\mathfrak{sl}_2$-triple $(e_\alpha, f_\alpha, h)$ in $\Fg_{\mathrm{NS}}$.
We normalize $h$ so that it acts on $H^{2i}(X)$ by multiplication by $i-n$.

We regard $h$ as a natural replacement for the operator $[N]^*$ in the abelian variety case. Our first result decomposes the Chow motive $\Fh(X)$ into eigenmotives of~$h$.

\begin{thm} \label{thm:Decomposition}
There is a unique Chow--K\"unneth decomposition:
\begin{equation} \label{ckdec}
\mathfrak{h}(X) = \bigoplus_{i = 0}^{2n} \Fh^{2i}(X)
\end{equation}
such that $h$ acts on $\Fh^{2i}(X)$ by multiplication by $i - n$.
\end{thm}

The mutually orthogonal projectors in the decomposition  \eqref{ckdec} are written explicitly in terms of the Heisenberg algebra action \cite{Groj, Nak}. We also show that~\eqref{ckdec} agrees with the Chow--K\"unneth decomposition obtained by de Cataldo--Migliorini~\cite{dCM} and Vial \cite{Vial}.

As before the decomposition \eqref{ckdec} specializes to a decomposition in Chow:
\begin{equation} \label{chowdec}
A^*(X) = \bigoplus_{i, s} A^i(X)_{2s}
\end{equation}
with:
\[A^i(X)_{2s} = A^i(\Fh^{2i - 2s}(X)) = \{\alpha \in A^i(X) \,|\, h(\alpha) = (i - s - n) \alpha\}.\]

\subsection{Multiplicativity}
\label{sub:mult}

In the seminal paper \cite{Be}, Beauville raised the question of whether hyper-K\"ahler varieties behave similarly to abelian varieties in the sense that the conjectural Bloch--Beilinson filtration also admits a multiplicative splitting. As a test case, he conjectured that for a hyper-K\"ahler variety, the cycle class map is injective on the subring generated by divisor classes.

For the Hilbert scheme of points of a K3 surface, Beauville's conjecture was recently proven in \cite{MN}; see also~\cite{Ob} for a shorter proof. But the ultimate goal remains to find the multiplicative splitting. Meanwhile, Shen and Vial \cite{SV, SV2} introduced the notion of a \emph{multiplicative Chow--K\"unneth decomposition}, upgrading Beauville's question from Chow groups to the level of correspondences/Chow motives.

The main result of this paper confirms that \eqref{ckdec} provides a multiplicative Chow--K\"unneth decomposition for the Hilbert scheme.

\begin{thm} \label{main}
Let $S$ be a projective K3 surface and let $X = \eHilb_n(S)$.
\begin{enumerate}
\item The Chow--K\"unneth decomposition \eqref{ckdec} is multiplicative, \emph{i.e.}, the cup product:
\[\cup: \Fh(X) \otimes \Fh(X) \to \Fh(X)\]
respects the grading, in the sense that: 
\[\cup: \Fh^{2i}(X) \otimes \Fh^{2j}(X) \to \Fh^{2i + 2j}(X)\]
for all $i, j \in \{0, ..., 2n\}$. As a result, the bigrading in \eqref{chowdec} is multiplicative.

\item All divisor classes and Chern classes of $X$ belong to $A^*(X)_0$.
\end{enumerate}
\end{thm}

Part (ii) of Theorem \ref{main} is related to the Beauville--Voisin conjecture \cite{Voi}, which predicts that for a hyper-K\"ahler variety, the cycle class map is injective on the subring generated by divisor classes and Chern classes. In the Hilbert scheme case, one may further ask the vanishing $A^*(X)_{2s} = 0$ for $s < 0$ and the injectivity of the cycle class map:
\[\mathrm{cl}: A^*(X)_0 \to H^*(X).\]
We do not tackle these questions in the present paper. 

The key to the proof of Theorem \ref{main} (i) is the compatibility between the grading operator $h$ and the cup product. For example, at the level of Chow groups, we show that the operator:
\[\widetilde{h} = h + n \Delta_X \in A^{2n}(X \times X)\]
acts on $A^*(X)$ by derivations, \emph{i.e.}:
\begin{equation}
\label{eqn:derivation}
\widetilde{h}(x \cdot x') = \widetilde{h}(x) \cdot x' + x \cdot \widetilde{h}(x')
\end{equation}
for all $x, x' \in A^*(X)$. We achieve this by explicit calculations using the Chow lifts~\cite{MN} of the well-known machinery for the Heisenberg algebra action~\cite{Lehn, LQW}, and our argument yields \eqref{eqn:derivation} at the level of correspondences; see Section \ref{sec:mult}. Once the compatibility is established, Theorem \ref{main} (i) is deduced by simply comparing the eigenvalues of $\widetilde{h}$.

\subsection{Previous work}
Theorem \ref{main} was previously obtained by Vial \cite{Vial} based on Voisin's announced result \cite[Theorem 5.12]{Voi2} on \emph{universally defined cycles}. A second proof, also relying on Voisin's theorem, was given by Fu and Tian~\cite{FT}. They interpreted Theorem \ref{main} (i) as the motivic incarnation of Ruan's crepant resolution conjecture~\cite{Ru}. Our proof has the advantage of being explicit and unconditional at the moment.

We note that multiplicative Chow--K\"unneth decompositions, for both hyper-K\"ahler and non-hyper-K\"ahler varieties, have been studied in \cite{FLV, FLV2, FT0, FTV, FV, FV2, LV}.

\subsection{Refined decomposition}

We further obtain a refined decomposition of the Chow motive $\Fh(X)$ with respect to the action of the N\'eron--Severi part of the LLV algebra~$\Fg_{\mathrm{NS}}$. Both the statement and the proof parallel the abelian variety case.

\begin{thm} \label{thm:redec}
Let $S$ be a projective K3 surface and let $X = \eHilb_n(S)$. There is a unique decomposition:
\begin{equation} \label{redec}
\Fh(X) = \bigoplus_{\psi \in \mathrm{Irrep}(\Fg_{\mathrm{NS}})}\Fh_\psi(X)
\end{equation}
where $\psi$ runs through all isomorphism classes of finite-dimensional irreducible representations of $\Fg_{\mathrm{NS}}$, and $\Fh_\psi(X)$ is $\psi$-isotypic under $\gNS$.
\end{thm}

Consider the weight decomposition:
\[\Fg_{\mathrm{NS}} = \Fg_{\mathrm{NS}, -2} \oplus \Fg_{\mathrm{NS}, 0} \oplus \Fg_{\mathrm{NS}, 2}, \quad \Fg_{\mathrm{NS}, 0} = \overline{\Fg}_{\mathrm{NS}} \oplus \BQ \cdot h\]
where $\overline{\Fg}_{\mathrm{NS}}$ is the N\'eron--Severi part of the \emph{reduced LLV algebra} (terminology taken from \cite{GKLR}). Let $\overline{\Ft} \subset \overline{\Fg}_{\mathrm{NS}}$ be a Cartan subalgebra and write:
\[ \Ft = \overline{\Ft} \oplus \BQ \cdot \widetilde{h}.\]
Then the decomposition  \eqref{redec} implies a motivic decomposition in terms of the irreducible representations (\emph{i.e.}, characters) of $\Ft$:
\begin{equation} \label{redec cartan}
\Fh(X) = \bigoplus_{\lambda \in \Ft^*} \Fh_\lambda(X).
\end{equation}
The decomposition \eqref{redec cartan} specializes to a refined decomposition in Chow:
\begin{equation} \label{refined in chow}
A^{\ast}(X) = \bigoplus_{i,s, \mu \in \overline{\Ft}^{\ast}} A^i(X)_{2s,\mu}
\end{equation}
where (with $\lambda = (i - s) \widetilde{h}^{\ast} + \mu$) we let:
\[
A^i(X)_{2s,\mu} = A^i(\Fh_{\lambda}(X)) = \{\alpha \in A^i(X)_{2s} \,|\, h_{v}(\alpha) = \mu(v) \alpha \text{ for all } v \in \overline{\Ft} \}.
\]

\begin{thm}
\label{thm:ref} \leavevmode As above, let $X = \emph{Hilb}_n(S)$ with $S$ a projective K3 surface.  

\begin{enumerate}
\item The decomposition \eqref{redec cartan} is multiplicative, \emph{i.e.}, the cup product respects the weight decomposition, in the sense that:
\begin{equation}
\label{eqn:cup}
\cup: \Fh_\lambda(X) \otimes \Fh_\mu(X) \to \Fh_{\lambda + \mu}(X)
\end{equation}
for all $\lambda, \mu \in \Ft^*$. 
As a result, the triple grading in \eqref{refined in chow} is multiplicative.

\item All Chern classes of $X$ belong to $A^{\ast}(X)_{0, 0}$.
\end{enumerate}
\end{thm} 

By a result of Markman \cite{Markman}, $\overline{\Fg}$ (namely the entire reduced LLV algebra, instead of just its N\'eron-Severi part; see \cite{GKLR}) is the Lie algebra of the monodromy group of~$X$. Since the images of Chern classes in cohomology are invariant under monodromy,
they are of weight $0$ with respect to $\overline{\Ft}$.
We see that part (ii) of Theorem~\ref{thm:ref} confirms the expectation from the Beauville--Voisin conjecture.

The decompositions \eqref{redec cartan} and \eqref{refined in chow} can be defined also for abelian varieties starting from the decomposition~\eqref{modec} obtained by Moonen.
The analogue of Theorem~\ref{thm:ref} (i) then follows from \cite[Proposition 6.9 (iii)]{Mon}.

Since
$\Ft \subset \overline{\Fg}_{\mathrm{NS}} \oplus \BQ \cdot \widetilde{h}$, Theorem \ref{thm:ref} (i) follows from the statement that any element of $\overline{\Fg}_{\mathrm{NS}}$ acts on $A^*(X)$ by derivations, akin to \eqref{eqn:derivation}. The generators of $\overline{\Fg}_{\mathrm{NS}}$ are denoted by:
$$h_{\alpha \beta} \in A^{2n}(X \times X)$$
and indexed by $\alpha \wedge \beta$ in $\wedge^2(A^1(X))$.
We then deduce Theorem \ref{thm:ref} (i) from the identity:
\begin{equation}
\label{eqn:derivation ref}
h_{\alpha\beta}(x \cdot x') = h_{\alpha\beta}(x) \cdot x' + x \cdot h_{\alpha\beta}(x') 
\end{equation}
for all $\alpha, \beta \in A^1(X)$ (our proof of \eqref{eqn:derivation ref} will be at the level of correspondences; see Section \ref{sec:mult}).

It is natural to ask for an extension of our results to arbitrary hyper-K\"ahler varieties.
By a result of Rie{\ss} \cite{Rie}, the Chow motives of two birational hyper-K\"ahler varieties are isomorphic as graded algebra objects. Moreover the isomorphism preserves Chern classes. Hence our results here apply equally well to any hyper-K\"ahler variety birational to the Hilbert scheme of points of a K3 surface.

\subsection{Conventions}
\label{sub:conv}

Throughout the present paper, Chow groups and Chow motives will be taken with $\BQ$-coefficients. We refer to \cite{MNP} for the definitions and conventions of Chow motives.

We will often switch between the languages of correspondences and operators on Chow groups, in the following sense. Every operator $f : A^*(X) \rightarrow A^*(Y)$ will arise from a correspondence $F \in A^*(X \times Y)$ by the usual construction:
$$
\xymatrix{& X \times Y \ar[ld]_{\pi_1} \ar[rd]^{\pi_2} & \\ X & & Y} \qquad f = \pi_{2*}(F \cdot \pi_1^*)
$$ 
and any compositions and equalities of operators implicitly entail compositions and equalities of correspondences. For example, the operator:
$$\mult_\tau : A^*(X) \rightarrow A^*(X)$$
of cup product with a fixed element $\tau \in A^*(X)$ is associated to the correspondence~$\Delta_*(\tau) \in A^*(X \times X)$, where $\Delta : X \hookrightarrow X \times X$ is the diagonal embedding.

Moreover, a family of operators $f_\gamma : A^*(X) \rightarrow A^*(Y)$ labeled by $\gamma \in A^*(Z)$ will arise from a correspondence $F \in A^*(X \times Y \times Z)$, by the assignment:
$$
f_\gamma \text{ arises from } \pi_{12*}(F \cdot \pi_3^*(\gamma)) \in A^*(X \times Y)
$$
for all $\gamma \in A^*(Z)$. We employ the language of ``operators indexed by $\gamma \in A^*(Z)$" instead of cycles on $X \times Y \times Z$ because it makes manifest the fact that $\gamma$ does not play any role in taking compositions. For instance, the family of operators:
$$
\mult_\gamma : A^*(X) \rightarrow A^*(X)
$$
labeled by $\gamma \in A^*(X)$ is associated to the small diagonal $\Delta_{123} \subset X \times X \times X$.

We will often be concerned with cycles on a variety of the form $S^n = S \times ... \times S$ for a smooth algebraic variety $S$ (most often an algebraic surface). We let:
$$
\Delta_{a_1 ... a_k} \in A^*(S^n)
$$
denote the diagonal $\{(x_1, ..., x_n) \,|\, x_{a_1} = ... = x_{a_k}\}$, for all collections of distinct indices $a_1, ..., a_k \in \{1, ..., n\}$. Moreover, given a class $\Gamma \in A^*(S^k)$, we may choose to write it as $\Gamma_{1 ... k}$ in order to indicate the power of $S$ where this class lives. Then for any collection of distinct indices $a_1,...,a_k \in \{1,...,n\}$, we define:
$$
\Gamma_{a_1 ... a_k} = p_{a_1 ... a_k}^{\ast}(\Gamma) \in A^*(S^n)
$$
where we let $p_{a_1 ... a_k} = (p_{a_1}, ..., p_{a_k}): S^n \to S^k$ with $p_i : S^n \to S$ the projection to the $i$-th factor. Finally, if $\bullet$ denotes any index from $1$ to $k+1$, we write:
$$
\int_\bullet : A^*(S^{k+1}) \rightarrow A^*(S^k)
$$
for the push-forward map which forgets the factor labeled by $\bullet$.

\subsection{Acknowledgements}

We would like to thank Lie Fu, Alina Marian, Davesh Maulik, Junliang Shen, Catharina Stroppel, and Zhiyu Tian for useful discussions. A.~N.~gratefully acknowledges the NSF grants DMS--1760264 and DMS--1845034, as well as support from the Alfred P.~Sloan Foundation. Q.~Y.~was supported by the NSFC grants 11701014, 11831013, and 11890661.

\section{Hilbert schemes}

\subsection{} Throughout the present paper, $S$ will denote a projective K3 surface. In \cite{BV}, Beauville and Voisin studied the class $c \in A^2(S)$ of any closed point on a rational curve in $S$, and they proved the following formulas in $A^*(S)$:
\begin{gather}
\label{eqn:bv 1}
c_2(\Tan_S) = 24c \\
\label{eqn:bv 2}
\alpha \cdot \beta = ( \alpha, \beta ) c
\end{gather}
for all $\alpha, \beta \in A^1(S)$ (above, we write $( \cdot, \cdot ) : A^*(S) \otimes A^*(S) \rightarrow \BQ$ for the intersection pairing). Moreover, we have the following identities in $A^*(S^2)$:
\begin{gather}
\label{eqn:bv 3}
\Delta \cdot c_1 = \Delta \cdot c_2 = c_1 \cdot c_2 \\
\label{eqn:bv 4}
\Delta \cdot \alpha_1 = \Delta \cdot \alpha_2 = \alpha_1 \cdot c_2 + \alpha_2 \cdot c_1 
\end{gather}
where $\Delta \in A^*(S^2)$ is the class of the diagonal, and the following identity in $A^*(S^3)$:
\begin{equation}
\label{eqn:bv 5}
\Delta_{123} = \Delta_{12} \cdot c_3 + \Delta_{13} \cdot c_2 + \Delta_{23} \cdot c_1 - c_1 \cdot c_2 - c_1 \cdot c_3 - c_2 \cdot c_3 .
\end{equation}
By iterating this identity, we obtain the corollary:
\begin{equation}
\label{eqn:bv 6}
\Delta_{1...k} = \sum_{1 \leq i < j \leq k} \Delta_{ij} \prod_{\ell \neq i,j} c_\ell - (k-2) \sum_{i=1}^k \prod_{\ell \neq i} c_\ell.
\end{equation}

\begin{prop}
\label{formulas}

The following formulas hold:
\begin{gather}
\label{eqn:new identity 1}
\gamma_1 c_1 = c_1 \int_\bullet \gamma_\bullet c_\bullet \\
\label{eqn:new identity 2}
\gamma_1 \alpha_1 = c_1 \int_\bullet \gamma_\bullet \alpha_\bullet + \alpha_1 \int_\bullet \gamma_\bullet c_\bullet \\
\label{eqn:new identity 3}
\quad\gamma_1 \Delta_{1...k} = \sum_{i=1}^k \gamma_i \prod_{j \neq i} c_j + \left(\Delta_{1...k} - \sum_{i=1}^k \prod_{j \neq i }c_j \right) \int_\bullet c_\bullet \gamma_\bullet - (k-1) c_1 ... c_k \int_\bullet \gamma_\bullet 
\end{gather} 
for any $\gamma \in A^*(S \times S^l)$, where only the first index of $\gamma$ appears in the equations above (the latter $l$ indices are simpy the same on the left and right-hand sides).

\end{prop}

\begin{proof} Formulas \eqref{eqn:new identity 1} and \eqref{eqn:new identity 2} both follow by taking \eqref{eqn:bv 3} and \eqref{eqn:bv 4} in~$A^*(S \times S)$ (with the factors denoted by indices $1$ and $\bullet$), multiplying them by $\gamma_\bullet$, and then integrating out the factor $\bullet$. As for \eqref{eqn:new identity 3}, let us consider identity \eqref{eqn:bv 5} in $A^*(S \times S \times S)$ (with the factors denoted by indices $1$, $2$, and $\bullet$) and multiply it by~$\gamma_\bullet$. We obtain:
\begin{align*}
\Delta_{12\bullet}\gamma_\bullet 
& = \Delta_{12} c_\bullet \gamma_\bullet + \Delta_{1\bullet} c_2 \gamma_\bullet + \Delta_{2\bullet} c_1 \gamma_\bullet - c_1 c_2 \gamma_\bullet - (c_1+c_2)c_\bullet \gamma_\bullet \\
\Rightarrow \Delta_{12\bullet}\gamma_1 & = \Delta_{12} c_\bullet \gamma_\bullet + \Delta_{1\bullet} c_2 \gamma_1 + \Delta_{2\bullet} c_1 \gamma_2 - c_1 c_2 \gamma_\bullet - (c_1+c_2)c_\bullet \gamma_\bullet.
\end{align*}
If we integrate out the factor $\bullet$, we precisely obtain the $k=2$ case of~\eqref{eqn:new identity 3}. To prove the general case of \eqref{eqn:new identity 3} we proceed by induction on $k$: the induction step is obtained by multiplying both sides of \eqref{eqn:new identity 3} with $\Delta_{k,k+1}$, and then applying \eqref{eqn:bv 3},~\eqref{eqn:bv 4}, and the $k=2$ case of \eqref{eqn:new identity 3}.
\end{proof}

\subsection{} 
\label{sub:ooo}

Consider the Hilbert scheme $\Hilb_n$ of $n$ points on $S$ and the Chow rings:
$$
\Hilb = \bigsqcup_{n=0}^\infty  \Hilb_n, \quad A^*(\Hilb) = \bigoplus_{n=0}^\infty A^*(\Hilb_n) 
$$
always with rational coefficients. We will consider two types of elements of the Chow rings above. The first of these are defined by considering the universal subscheme:
$$
\CZ_n \subset \Hilb_n \times S.
$$
For any $k \in \BN$, consider the projections:
\[ \begin{tikzcd}
\Hilb_n & \Hilb_n \times S^k \ar[swap]{l}{\pi} \ar{r}{\rho} & S^k
\end{tikzcd} \]
and let $\CZ_n^{(i)} \subset \Hilb_n \times S^k$ denote the pull-back of $\CZ_n$ via the~$i$-th projection $S^k \rightarrow S$. 

\begin{defn}
\label{def:universal}
	
A \emph{universal class} is any element of $A^*(\Hilb_n)$ of the form:
\begin{equation}
\label{eqn:tautological}
\pi_{*} \Big[ P(...,\ch_{j}(\CO_{\CZ^{(i)}_n}),...)^{1 \leq i \leq k}_{j \in \BN} \Big]
\end{equation}
for all $k \in \BN$ and for all polynomials $P$ with coefficients pulled back from $A^*(S^k)$.

\end{defn}

In particular, by inserting diagonals if necessary,
the classes \eqref{eqn:tautological} where $P$ is a monomial are of the form:
\begin{equation}
\label{eqn:formula taut 1}
\univ_{d_1,...,d_k}(\Gamma) = \pi_{*} \Big[ \ch_{d_1}(\CO_{\CZ^{(1)}_n}) ... \ch_{d_k}(\CO_{\CZ^{(k)}_n}) \cdot \rho^*(\Gamma) \Big].
\end{equation}
The following theorem holds for every smooth quasi-projective surface (see \cite{N taut}), but we only prove it here in the case where $S$ is a K3 surface (the argument herein easily generalizes to any smooth projective surface using the results of \cite{GT}). 

\begin{thm}
\label{thm:taut}

Any class in $A^*(\eHilb_n)$ is universal, \emph{i.e.}, of the form \eqref{eqn:tautological}.

\end{thm}

\begin{proof} Consider the product $\Hilb_n \times S^k \times \Hilb_n$, and we will write $\pi_1$, $\pi_2$, $\pi_3$, $\pi_{12}$, $\pi_{23}$, and $\pi_{13}$ for the various projections to its factors. As a consequence of \cite{M} (see also \cite{GT}), the diagonal $\Delta_{\Hilb_n} \subset \Hilb_n \times \Hilb_n$ can be written as follows: 
$$
\Delta_{\Hilb_n} = \pi_{13*} \left[\sum_a \pi_2^*(\gamma_a) \prod_{(i,j)} \ch_j \left( \CO_{\CZ_n^{(i)}} \right) \prod_{(\widetilde{i}, \widetilde{j})} \ch_{\widetilde{j}} \left( \CO_{\widetilde{\CZ}_n^{(\widetilde{i})}} \right) \right]
$$
for suitably chosen $k \in \BN$, where we do not care much about the specific coefficients~$\gamma_a$ and indices $i,j,\widetilde{i}, \widetilde{j}$ which appear in the sum above (we write $\CZ_n$ and~$\widetilde{\CZ}_n$ for the universal subschemes in $\Hilb_n \times S \times \Hilb_n$ corresponding to the first and second copies of $\Hilb_n$, respectively). Since the diagonal corresponds to the identity operator, the equality above implies that:
\begin{align}
\text{Id}_{\Hilb_n} & = \pi_{1*} \left[\sum_a \pi_2^*(\gamma_a) \prod_{(i,j)} \ch_j \left( \CO_{\CZ_n^{(i)}} \right) \prod_{(\widetilde{i}, \widetilde{j})} \ch_{\widetilde{j}} \left( \CO_{\widetilde{\CZ}_n^{(\widetilde{i})}} \right) \pi_3^* \right] \\
\label{eqn:last}
& = \sum_a \pi_* \left[ \prod_{(i,j)} \ch_j \left( \CO_{\CZ_n^{(i)}} \right) \rho^*\left( \gamma_a \cdot \rho_* \left( \prod_{(\widetilde{i}, \widetilde{j})} \ch_{\widetilde{j}} \left( \CO_{\widetilde{\CZ}_n^{(\widetilde{i})}} \right)  \cdot \pi^* \right) \right) \right]
\end{align}
hence the universality.
\end{proof}

Formula \eqref{eqn:last} implies the surjectivity of the homomorphism:
\begin{equation}
\label{eqn:surj}
\begin{gathered}
\bigoplus_a A^*(S^k) \twoheadrightarrow A^*(\Hilb_n) \\
\sum_a \Gamma_a \mapsto \sum_a \pi_* \left[ \prod_{(i,j)} \ch_j \left( \CO_{\CZ_n^{(i)}} \right) \rho^*(\Gamma_a) \right]
\end{gathered}
\end{equation}
where the sums over $a$ are in one-to-one correspondence with the sums in \eqref{eqn:last}. 

\subsection{}
\label{sub:nakajima}

Let us present another important source of elements of $A^*(\Hilb_n)$, based on the following construction independently due to Grojnowski \cite{Groj} and Nakajima~\cite{Nak} (in the present paper, we will mostly use the presentation by Nakajima). For any~$n,k \in \BN$, consider the closed subscheme:
$$
\Hilb_{n,n+k} = \Big\{(I \supset I') \,|\, I/I' \text{ is supported at a single }x \in S \Big\} \subset \Hilb_n \times \Hilb_{n+k}
$$
endowed with projection maps:
\begin{equation}
\label{eqn:diagram zk}
\xymatrix{& \Hilb_{n,n+k} \ar[ld]_{p_-} \ar[d]^{p_S} \ar[rd]^{p_+} & \\ \Hilb_{n} & S & \Hilb_{n+k}}
\end{equation}
that remember $I$, $x$, $I'$, respectively. One may use $\Hilb_{n,n+k}$ as a correspondence:
\begin{equation}
\label{eqn:nakajima}
A^*(\Hilb_n) \xrightarrow{\fq_{\pm k}} A^*(\Hilb_{n \pm k} \times S)
\end{equation}
given by:
\begin{equation}
\label{eqn:nak def}
\fq_{\pm k} = (\pm 1)^{k} \cdot (p_\pm \times p_S)_* \circ p_\mp^*.
\end{equation}
Because the correspondences above are defined for all $n$, it makes sense to set:
$$
A^*(\Hilb) \xrightarrow{\fq_{\pm k}} A^*(\Hilb \times S).
$$
We also set $\fq_0 = 0$. The main result of \cite{Nak} (although \emph{loc.~cit.}~is written at the level of cohomology, the result holds at the level of Chow groups; see for example \cite[Remark 8.15 (2)]{Nak lectures}) is that the operators $\fq_k$ obey the commutation relations in the Heisenberg algebra, namely: 
\begin{equation}
\label{eqn:heis}
[\fq_k, \fq_l] = k \delta_{k+l}^0 \left( \text{Id}_{\Hilb} \times \Delta \right)
\end{equation}
as correspondences $A^*(\Hilb) \rightarrow A^*(\Hilb \times S^2)$. In terms of self-correspondences $A^*(\Hilb) \rightarrow A^*(\Hilb)$, the identity \eqref{eqn:heis} reads, for all $\alpha, \beta \in A^*(S)$:
\begin{equation}
\label{eqn:heis op}
[\fq_k(\alpha), \fq_l(\beta)] 
=  k ( \alpha, \beta ) \Id_\Hilb. 
\end{equation}

\subsection{} More generally, we may consider:
\begin{equation}
\label{eqn:composition 1}
\fq_{n_1}...\fq_{n_t} : A^*(\Hilb) \rightarrow A^*(\Hilb \times S^t)
\end{equation}
where the convention is that the operator $\fq_{n_i}$ acts in the $i$-th factor of $S^t = S \times ... \times S$. Then associated to any $\Gamma \in A^*(S^t)$, one obtains an endomorphism of $A^*(\Hilb)$:
\begin{equation}
\label{eqn:composition 2}
\fq_{n_1}...\fq_{n_t}(\Gamma) = \pi_{*} (\rho^*(\Gamma) \cdot \fq_{n_1}...\fq_{n_t})
\end{equation}
where $\pi$ and $\rho$ denote the projections of $\Hilb \times S^t$ to the factors.

\begin{thm}[\cite{dCM}]
\label{thm:dcm}

We have a decomposition:
\begin{equation}
\label{eqn:decomp}
A^*(\eHilb) = \bigoplus_{\substack{n_1 \geq ... \geq n_t \in \BN \\ \Gamma \in A^*(S^t)^{\emph{sym}}}}
\fq_{n_1}... \fq_{n_t}(\Gamma) \cdot v
\end{equation}
where ``\emph{sym}" refers to the part of $A^*(S^t)$ which is symmetric with respect to those transpositions $(ij) \in \mathfrak{S}_t$ for which $n_i = n_j$, and $v$ is a generator of $A^*(\eHilb_0) \cong \BQ$. 
\end{thm}

\begin{proof}
Since we will need it later, we recall the precise relationship between Nakajima operators
and the correspondences studied in \cite{dCM}.
Let $\lambda$ be a partition of $n$ with $k$ parts, let $S^{\lambda} = S^k$ and
let $S^{\lambda} \to S^{(n)}$ be the map that sends $(x_1, ..., x_k)$ to the cycle $\lambda_1 x_1 + ... + \lambda_k x_k$ in
the $n$-th symmetric product of the surface $S$. We consider the correspondence:
\begin{align*}
\Gamma_{\lambda} 
& = (\Hilb_n \times_{S^{(n)}} S^{\lambda})_{\text{red}} \\
& = \{ (I,x_1, ..., x_{k}) \,|\, \sigma(I) = \lambda_1 x_1 + ... + \lambda_{k} x_{k} \} 
\end{align*} 
where 
$\sigma : S^{[n]} \to S^{(n)}$ is the Hilbert--Chow morphism
sending the subscheme $I$ to its underlying support.
The subscheme $\Gamma_{\lambda}$ is irreducible of dimension $n+k$
and the locus $\Gamma_{\lambda}^{\text{reg}} \subset \Gamma_{\lambda}$,
where the points $x_i$ are distinct, is open and dense; see \cite[Remark~2.0.1]{dCM}.
Similarly, the Nakajima correspondence $\Fq_{\lambda_1} ... \Fq_{\lambda_k}$ is a cycle in~$\Hilb_n \times S^k$ of dimension $n+k$
supported on a subscheme that contains $\Gamma_{\lambda}^{\text{reg}}$ as an open subset and
whose complement is of smaller dimension \cite[4 (i)]{Nak}.
Moreover the multiplicity of the cycle on $\Gamma_{\lambda}^{\text{reg}}$ is $1$.
Hence we have the equality of correspondences:
\begin{equation} \Gamma_{\lambda} = \Fq_{\lambda_1} ... \Fq_{\lambda_k} \in A^{\ast}(\Hilb_n \times S^{\lambda}). \label{nak=dCM} \end{equation}
The result follows now from \cite[Proposition 6.1.5]{dCM}, which says that:
\begin{equation} \Delta_{\Hilb_n} = \sum_{\lambda \vdash n}  \frac{(-1)^{n-l(\lambda)}}{|\Aut(\lambda)|\prod_i \lambda_i} {^t\Gamma_{\lambda}} \circ \Gamma_{\lambda} \label{diagonal_eqn} 
\end{equation}
where $\lambda$ runs over all partitions of size $n$, and we let $l(\lambda)$ and $\lambda_i$ denote the length and the parts of $\lambda$, respectively.
\end{proof}

\begin{rmk} As shown in \cite{N taut}, there is an explicit way to go between the descriptions~\eqref{eqn:tautological} and \eqref{eqn:decomp} of $A^*(\Hilb)$. Concretely, for all $n_1 \geq ... \geq n_t$ there exists a polynomial $P_{n_1,...,n_t}$ with coefficients in $\rho^*(A^*(S^t))$ such that for all $\Gamma \in A^{\ast}(S^t)$:
\begin{equation}
\label{eqn:connection}
\fq_{n_1}...\fq_{n_t}(\Gamma) = \pi_* \Big[ P_{n_1,...,n_t}(...,\ch_{j}(\CO_{\CZ^{(i)}}),...)^{1 \leq i \leq t}_{j \in \BN} \cdot  \rho^*(\Gamma) \Big].
\end{equation}
Moreover, \emph{loc.~cit.}~gives an algorithm for computing the polynomial $P_{n_1,...,n_t}$. 

\end{rmk}

\subsection{} 
\label{sub:chern}
Two interesting collections of elements of $A^*(\Hilb_n)$ can be written as universal classes: divisors and Chern classes of the tangent bundle. 
One has:
$$
A^1(S) \oplus \BQ \cdot \delta \cong A^1(\Hilb_n)
$$
(with the convention that $\delta = 0$ if $n=1$) where the isomorphism is given by:
\begin{align}
\label{eqn:div 1}
l \in A^1(S) & \ \mapsto \  \univ_2(l) \\
\label{eqn:div 2}
\delta & \ \mapsto \  \univ_3(1).
\end{align}
Similarly, the Chern character of the tangent bundle to $\Hilb_n$ is given by the well-known formula (see for example \cite[Proposition 2.10]{MN}):
$$
\ch(\Tan_{\Hilb_n}) = \pi_* \left[ \left(\ch \left(\CO_{\CZ_n} \right) + \ch \left(\CO_{\CZ_n} \right)' - \ch \left(\CO_{\CZ_n} \right) \ch \left(\CO_{\CZ_n} \right)' \right) \rho^*(1+2c)\right]
$$
where $( \text{ })'$ is the operator which multiplies a degree $d$ class by $(-1)^d$. Therefore, the Chern character of the tangent bundle is a linear combination of the following particular universal classes:
\begin{equation}
\label{eqn:tan}
\univ_d(\gamma) \quad \text{ and } \quad \univ_{d,d'}(\Delta_*(\gamma))
\end{equation}
where $\gamma \in \{1,c\}$, and $d,d'$ are various natural numbers.

\section{Motivic decompositions} 
\subsection{} \label{sec:LLV_algebra}
Let us recall the Lie algebra action $\fg_{\mathrm{NS}} \curvearrowright A^*(\Hilb_n)$ from \cite{Ob},
which lifts the classical construction of \cite{LL, V} in cohomology.
To this end, consider the Beauville--Bogomolov form, which is the pairing on:
$$
V = A^1(\Hilb_n) \cong A^1(S) \oplus \BQ \cdot \delta
$$
This form extends the intersection form on $A^1(S)$ and satisfies:
$$
(\delta, \delta) = 2 - 2n, \quad (\delta, A^1(S)) = 0.
$$
Let $U=(\begin{smallmatrix} 0& 1 \\1 &0\end{smallmatrix})$ be the hyperbolic lattice with fixed symplectic basis $e,f$.
We have:
\[ \Fg_{\mathrm{NS}} = \wedge^2 ( V \overset{\perp}{\oplus} U_{\BQ} ) \]
where the Lie bracket is defined for all $a,b,c,d \in V \oplus U_\BQ$ by:
\[ [a \wedge b, c \wedge d] = (a,d) b \wedge c - (a,c) b \wedge d - (b,d) a \wedge c + (b,c) a \wedge d. \]

Consider for all $\alpha \in A^1(S)$ the following operators:
\begin{gather}
e_{\alpha} = -\sum_{n > 0} \Fq_{n} \Fq_{-n} ( \Delta_{\ast} \alpha) \nonumber \\
\label{LLV_operators}
\begin{gathered}
e_{\delta} = -\frac{1}{6} \sum_{i+j+k=0} :\! \Fq_i \Fq_j \Fq_k ( \Delta_{123} )\!: \\
\widetilde{f}_{\alpha} = -\sum_{n > 0} \frac{1}{n^2} \Fq_{n} \Fq_{-n}( \alpha_1 + \alpha_2 )
\end{gathered} \\
\nonumber
\widetilde{f}_{\delta}
= -\frac{1}{6} \sum_{i+j+k=0} :\!\Fq_i \Fq_j \Fq_k \left( \frac{1}{k^2}  \Delta_{12} + \frac{1}{j^2} \Delta_{13} + \frac{1}{i^2} \Delta_{23} + \frac{2}{j k} c_1 + \frac{2}{i k} c_2 + \frac{2}{i j} c_3 \right)\!: \,.
\end{gather}
Here $:\! - \!:$ is the normal ordered product defined by:
\begin{equation}
\label{eqn:normal}
:\! \Fq_{i_1} ... \Fq_{i_k}\!: \,\, =  \Fq_{i_{\sigma(1)}} ... \Fq_{i_{\sigma(k)}}
\end{equation}
where $\sigma$ is any permutation such that $i_{\sigma(1)} \geq ... \geq i_{\sigma(k)}$. We define operators $e_{\alpha}$ and~$\widetilde{f}_\alpha$ for general $\alpha \in A^1(\Hilb_n)$ by linearity in $\alpha$.
By Theorem 1.6 of \cite{MN}, we have that $e_\alpha$ is the operator of cup product with $\alpha$.
If $(\alpha,\alpha) \neq 0$, the multiple~$\widetilde{f}_\alpha / (\alpha,\alpha)$ acts on cohomology as the Lefschetz dual of $e_\alpha$. In \cite{Ob}, it was shown that the assignment:
\begin{equation}
\label{eqn:action}
\begin{gathered}
\act : \Fg_{\mathrm{NS}} \rightarrow A^*(\Hilb_n \times \Hilb_n) \\
\act(e \wedge \alpha) = e_\alpha, \quad \act(\alpha \wedge f) = \widetilde{f}_\alpha
\end{gathered}
\end{equation}
for all $\alpha \in V$, induces a Lie algebra homomorphism. In particular, it was shown in \emph{loc.~cit.}~that the element~\mbox{$e \wedge f \in \gNS$} acts by:
\begin{equation} 
\label{eqn:def h} 
h = \sum_{k > 0} \frac{1}{k} \Fq_{k} \Fq_{-k}( c_2 - c_1 ).
\end{equation}
The operator $h$ specializes in cohomology to the Lefschetz grading operator, which by our normalization acts on $H^{2i}(\Hilb_n)$ by multiplication by $i-n$. Similarly, for any $\alpha, \beta \in A^1(S) \subset A^1(\Hilb_n)$, the element $\alpha \wedge \beta \in \gNS$ acts by:
\begin{equation} 
\label{eqn:def h ab} 
h_{\alpha \beta} = \sum_{k=1}^\infty \frac{1}{k} \fq_k \fq_{-k}(\alpha_2 \beta_1 - \alpha_1 \beta_2)
\end{equation}
and the element $\alpha \wedge \delta \in \gNS$ acts by:
\begin{equation} 
\label{eqn:def h ad} 
h_{\alpha \delta} = - \frac{1}{2} \sum_{i+j+k = 0, i,j,k \in \BZ} \frac{1}{k} :\!\fq_i \fq_j \fq_k(\Delta_{12}(\alpha_1+\alpha_3))\!:\,.
\end{equation}

\subsection{Proof of Theorem~\ref{thm:Decomposition}}
We start with the decomposition of the diagonal into Nakajima operators:
\begin{equation} \label{diagonal} \Delta_{\Hilb_n} = 
\sum_{\lambda \vdash n} \frac{(-1)^{l(\lambda)}}{\Fz(\lambda)} 
\Fq_{\lambda} \Fq_{-\lambda}(\Delta)
\end{equation} 
where $\lambda$ runs over all partitions of $n$,
\[ \Fz(\lambda) = |\mathrm{Aut}(\lambda)| \prod_{i} \lambda_i \]
is a combinatorial factor, and for any $\pi \in A^{\ast}(S \times S)$ we write:
\begin{align*} \Fq_{\lambda} \Fq_{-\lambda}(\pi) 
& =  \Fq_{\lambda_1} ... \Fq_{\lambda_{l(\lambda)}} \Fq_{-\lambda_1} ... \Fq_{-\lambda_{l(\lambda)}} \left( \pi_{1,l(\lambda)+1} \pi_{2,l(\lambda)+2} ... \pi_{l(\lambda),2l(\lambda)} \right)  \\
& =\,\,  : \! \Fq_{\lambda_1} \Fq_{-\lambda_1}(\pi) ... \Fq_{\lambda_{l(\lambda)}} \Fq_{-\lambda_{l(\lambda)}}(\pi)\! : \,.
\end{align*}
The formula \eqref{diagonal} follows directly from \eqref{nak=dCM}, \eqref{diagonal_eqn}, and the fact that ${^t\Fq_m} = (-1)^m \Fq_{-m}$ (which is incorporated in the definition \eqref{eqn:nak def}).

Consider the decomposition of the diagonal of $S$ as:
\begin{equation} \Delta = \pi_{-1} + \pi_0 + \pi_{1} \label{4234} \end{equation}
where:
\[ \pi_{-1} = c_1, \quad \pi_0 = \Delta - c_1 - c_2, \quad \pi_{1} = c_2. \]
It is easy to note that $\pi_{-1}$, $\pi_0$, $\pi_1$ are the projectors onto the $-1, 0, +1$ eigenspaces of the action of $h$ on $A^{\ast}(\Hilb_1) = A^{\ast}(S)$. 

To define projectors corresponding to the action of $h$ on $A^*(\Hilb_n)$, we insert the decomposition \eqref{4234} into \eqref{diagonal}, and then expand and collect the terms of degree~$i$. Concretely, for every integer $i$, we let:
\[
P_{i} = 
\sum^{\lambda, \mu, \nu}_{\substack{|\lambda| + |\mu| + |\nu| = n \\ -l(\lambda) + l(\nu) = i}}
\frac{(-1)^{l(\lambda) + l(\mu) + l(\nu)}}{\Fz(\lambda) \Fz(\mu) \Fz(\nu)}
: \! \Fq_{\lambda} \Fq_{-\lambda}({^{t}\pi_{-1}}) \Fq_{\mu} \Fq_{-\mu}({^{t}\pi_0}) \Fq_{\nu} \Fq_{-\nu}({^{t}\pi_{1}}) \! : \,.
\]
In particular, we have $P_i = 0$ unless $i \in \{-n, ..., n\}$. Let us check that $P_i$ are indeed projectors onto the eigenspaces of $h$.
\begin{claim} For all $i,j \in \{-n, ..., n\}$ we have the following equalities in $A^{\ast}(\eHilb_n \times \eHilb_n)$:
\begin{enumerate}
\item[(a)] $P_{i} \circ P_{j} = P_{i} \delta^i_{j}$
\item[(b)] $h \circ P_i = i P_i$.
\end{enumerate}
\end{claim}
\begin{proof}
(a) We determine $P_i \circ P_j$ by commuting all Nakajima operators with negative indices to the right, and then using that
we act on $\Hilb_n$ so all products of Nakajima operators with purely negative indices of degree $>n$ vanish.
Since every summand in $P_j$ contains such a product of degree $n$,
we find that for a term to contribute all operators with negative indices coming from $P_i$ have to interact with operators (with positive indices) from the second term.
The interactions are described as follows. For a single term (let $a,b > 0$ and $r,s \in \{ -1, 0, 1 \}$) we have:
\begin{multline*}
\Fq_{a} \Fq_{-a}(^t\pi_{r}) \Fq_{b} \Fq_{-b}(^t\pi_{s}) 
= \Fq_{a} [ \Fq_{-a} , \Fq_{b} ] \Fq_{-b}\left( (^t\pi_{r})_{12} (^t\pi_{s})_{34} \right) \\
+ \Fq_{a} \Fq_{b} \Fq_{-a} \Fq_{-b} ( (^t\pi_{r})_{13} (^t\pi_{s})_{24} )
\end{multline*}
where by the commutation relations \eqref{eqn:heis} the first term on the right is:
\begin{align*}
\Fq_{a} [ \Fq_{-a} , \Fq_{b} ] \Fq_{-b}\left( (^t\pi_{r})_{12} (^t\pi_{s})_{34} \right)
& = (-a) \delta_{ab} \Fq_a \Fq_{-b} \left( \pi_{14 \ast}( (^t\pi_{r})_{12} (^t\pi_{s})_{34} \Delta_{23} ) \right)\\
& = (-a) \delta_{ab} \Fq_a \Fq_{-a} \left( ^t\pi_{s} \circ {^t\pi_r} \right) \\
& = (-a) \delta_{ab} \Fq_a \Fq_{-a} \left( ^t(\pi_{r} \circ \pi_s) \right) \\
& = (-a) \delta_{ab} \delta_{rs} \Fq_a \Fq_{-a} (^t\pi_{r}).
\end{align*}
Hence for a composition:
\[
: \! \Fq_{\lambda} \Fq_{-\lambda}(^{t}\pi_{-1}) \Fq_{\mu} \Fq_{-\mu}(^{t}\pi_0) \Fq_{\nu} \Fq_{-\nu}(^{t}\pi_{1}) \! : 
\circ
: \! \Fq_{\lambda'} \Fq_{-\lambda'}(^{t}\pi_{-1}) \Fq_{\mu'} \Fq_{-\mu'}(^{t}\pi_0) \Fq_{\nu'} \Fq_{-\nu'}(^{t}\pi_{1}) \! :\\
\]
(with $\lambda, \mu, \nu$ as in the definition of $P_i$, and the same for the primed partitions)
to act non-trivially on $\Hilb_n$ we have to have $\lambda = \lambda'$, $\mu = \mu'$ and $\nu = \nu'$.
Moreover, if we write $\lambda$ multiplicatively as $(1^{l_1} 2^{l_2} ... )$
where $l_i$ is the number of parts of size $i$, then
there are precisely $|\Aut(\lambda)| = \prod_i l_i!$ different ways to pair the negative factors in
$\Fq_{\lambda} \Fq_{-\lambda}(^{t}\pi_{-1})$ with the positive factors $\Fq_{\lambda'} \Fq_{-\lambda'}(^{t}\pi_{-1})$,
and similarly for $\mu, \nu$.
Hence:
\begin{align*}
& : \! \Fq_{\lambda} \Fq_{-\lambda}(^{t}\pi_{-1}) \Fq_{\mu} \Fq_{-\mu}(^{t}\pi_0) \Fq_{\nu} \Fq_{-\nu}(^{t}\pi_{1}) \! : 
\circ
: \! \Fq_{\lambda'} \Fq_{-\lambda'}(^{t}\pi_{-1}) \Fq_{\mu'} \Fq_{-\mu'}(^{t}\pi_0) \Fq_{\nu'} \Fq_{-\nu'}(^{t}\pi_{1}) \! :\\
={} &
\delta_{\lambda \lambda'} \delta_{\mu \mu'}  \delta_{\nu \nu'}  (-1)^{l(\lambda) + l(\mu) + l(\nu)}\Fz(\lambda) \Fz(\mu) \Fz(\nu) 
 : \! \Fq_{\lambda} \Fq_{-\lambda}(^{t}\pi_{-1}) \Fq_{\mu} \Fq_{-\mu}(^{t}\pi_0) \Fq_{\nu} \Fq_{-\nu}(^{t}\pi_{1}) \! : 
\end{align*}
which implies the claim.

(b) To determine $h \circ P_i$ we commute $h$ into the middle, \emph{i.e.}, to the right of all Nakajima operators with positive indices, and to the left of all with negative ones.
In the middle position $h$ acts on the Chow ring of $\Hilb_0$, where it vanishes. Hence again we only need to compute the commutators. For this we use \eqref{eqn:comm formula 1} and that $\pi_i$ are the projectors onto the eigenspaces of $h$ so that:
\[ (h \times \Id)( ^t\pi_r ) = {^t((\Id \times h)(\pi_r))} = {^t(h \circ \pi_r)} = r (^t\pi_r). \]
As desired we find:
\[ h \circ P_i = (-1 \cdot l(\lambda) + 0 \cdot l(\mu) + 1 \cdot l(\nu)) P_i = i P_i. \qedhere \]
\end{proof}

Using the claim it follows that the motivic decomposition:
\[ \Fh(\Hilb_n) = \bigoplus_{i=0}^{2n} \Fh^{2i}(\Hilb_n) \]
with $\Fh^{2i}(\Hilb_n) = (\Hilb_n, P_{i-n})$ has the stated properties. The uniqueness of the decomposition follows from the uniqueness of the decomposition of $\Delta_{\Hilb_n}$ under the action of $h$ on $A^{\ast}(\Hilb_n \times \Hilb_n)$; see the proof for the refined decomposition in Section~\ref{sec:refdec} below.
\qed

\secskip
By \eqref{diagonal_eqn}, an alternative way to write the projector $P_i$ is:
\[ P_i = \sum_{\lambda \vdash n}
\frac{(-1)^{n-l(\lambda)}}{\Fz(\lambda)} {^t\Gamma_{\lambda}} \circ \widetilde{P}_i \circ \Gamma_{\lambda} \]
where $\widetilde{P}_i \in A^{\ast}(S^{\lambda} \times S^\lambda)$ is the projector:
\[ \widetilde{P}_i = \sum_{i_1 + ... + i_{l(\lambda)} = i} \pi_{i_1} \times ... \times \pi_{i_{l(\lambda)}}. \]
Hence
the decomposition of Theorem~\ref{thm:Decomposition} is precisely
the Chow--K\"unneth decomposition constructed by Vial in \cite[Section 2]{Vial}.

\subsection{Refined decomposition} \label{sec:refdec}

Let $U(\Fg_{\mathrm{NS}})$ be the universal enveloping algebra of~$\Fg_{\mathrm{NS}}$. The Lie algebra homomorphism \eqref{eqn:action} extends to an algebra homomorphism:
$$
\act : U(\Fg_{\text{NS}}) \rightarrow A^*(\Hilb_n \times \Hilb_n).
$$

\begin{lemma} \label{lemma_finite_dim}

The image $W \subset A^*(\eHilb_n \times \eHilb_n)$ of $\act$ is finite-dimensional.

\end{lemma}
\begin{proof}
For every fixed $k \geq 1$ the subring of $R^{\ast}(S^k) \subset A^{\ast}(S^k)$ generated by:
\begin{itemize}
\item $\alpha_i$ for all $i$ and $\alpha \in A^1(S)$
\item $c_i$ for all $i$
\item $\Delta_{ij}$ for all $i,j$
\end{itemize}
is finite-dimensional, and preserved by the projections to the factors. Hence the space of operators $\widetilde{W} \subset A^{\ast}(\Hilb_n \times \Hilb_n)$
spanned by:
\[ \Fq_{\lambda_1} ... \Fq_{\lambda_{l(\lambda)}} \Fq_{-\mu_1} ... \Fq_{-\mu_{l(\mu)}}(\Gamma) \]
for all partitions $\lambda, \mu$ of $n$ and all $\Gamma \in R^{\ast}(S^{l(\lambda) + l(\mu)})$ is finite-dimensional.
The commutation relations \eqref{eqn:heis} show that $\widetilde{W}$ is closed under compositions of correspondences.
Moreover, by inspecting the expressions for the generators of $\Fg_{\mathrm{NS}}$ in~\eqref{LLV_operators} (and using \eqref{diagonal}
to bring them into the desired form), we see that all generators of~$\gNS$ lie in $\widetilde{W}$.
Hence $g \in \widetilde{W}$ for all $g \in U(\gNS)$, \emph{i.e.}, $W \subset \widetilde{W}$.
\end{proof}

We find that $W$ is a finite-dimensional vector space which is preserved by the action of $U(\gNS)$,
and hence defines a finite-dimensional representation of $\gNS$.
Since~$\gNS$ is semisimple, this representation decomposes into isotypic summands:
\begin{equation} \label{sum is w}
W = \bigoplus_{\psi \in \mathrm{Irrep}(\Fg_{\mathrm{NS}})} W_{\psi}.
\end{equation}
Let us look at the image of $\Delta_{\Hilb_n} \in W$ under this decomposition:
\begin{equation}
\label{sum is 1}
\Delta_{\Hilb_n} = \sum_{\psi \in \mathrm{Irrep}(\Fg_{\mathrm{NS}})} P_{\psi}
\end{equation} 
where $P_\psi \in W_\psi$. 

\begin{claim}
\label{claim:XX}
The elements $P_{\psi} \in A^*(\eHilb_n \times \eHilb_n)$ are orthogonal projectors.
\end{claim}

\begin{proof} Let us first show that left-multiplication by $P_\psi$ maps $W$ to $W_\psi$, \emph{i.e.}:
\begin{equation}
\label{eqn:proj}
P_\psi \circ W \subset W_\psi.
\end{equation}
Indeed, for all $a \in W$, right multiplication by $a$ is a $\Fg_{\text{NS}}$-intertwiner and thus sends~$W_\psi$ to $W_\psi$. In other words, we have $W_\psi \circ a \subset W_\psi$, hence $W_\psi \circ W \subset W_\psi$, which implies \eqref{eqn:proj}. If we multiply any $a \in W$ by relation \eqref{sum is 1}, we obtain:
$$
a = \sum_{\psi \in \mathrm{Irrep}(\Fg_{\mathrm{NS}})} P_{\psi} \circ a.
$$
By \eqref{eqn:proj}, the summands in the right-hand side each lie in $W_\psi$. 
If 
$a \in W_{\psi'}$, 
then by comparing summands the equality above implies:
$$
P_{\psi'} \circ a  = a \quad \text{and} \quad P_{\psi} \circ a = 0
$$
for all $\psi \neq \psi'$. In particular, taking $a = P_{\psi'}$ implies the relations $P_{\psi} \circ P_{\psi'} = \delta_{\psi'}^{\psi} P_{\psi}$. Moreover, this implies that the inclusion \eqref{eqn:proj} is actually an identity, hence left multiplication by $P_\psi$ projects $W$ onto $W_\psi$. 
\end{proof}

From Claim \ref{claim:XX} we obtain the decomposition:
\begin{equation} \Fh(\Hilb_n) = \bigoplus_{\psi \in \mathrm{Irrep}(\Fg_{\mathrm{NS}})} \Fh_\psi(\Hilb_n) \label{sdas} \end{equation} 
where $\Fh_\psi(\Hilb_n) = (\Hilb_n, P_{\psi})$.
We can now prove the main result of this section.

\begin{proof}[Proof of Theorem~\ref{thm:redec}]
It remains to show that the summands $\Fh_\psi(\Hilb_n)$ are $\psi$-isotypic and that the decomposition \eqref{sdas} is unique.
Let $M$ be a Chow motive. The action of $\gNS$ on $\Hom(M, \Fh(\Hilb_n))$
is defined by
$g \mapsto \text{act}(g) \circ ( - )\,$. 
Hence if~$f\in \Hom(M, M')$ is a morphism of Chow motives, the pullback:
\[f^{\ast} : \Hom(M', \Fh(\Hilb_n)) \to \Hom(M, \Fh(\Hilb_n))\]
is equivariant with respect to the $\Fg_{\text{NS}}$-action. Now, for any $v \in \Hom(M, \Fh_{\psi}(\Hilb_n))$ we have $v = P_\psi \circ w$ for some $w \in \Hom(M, \Fh(\Hilb_n))$ and thus:
\[ U(\gNS) v = U(\gNS) w^{\ast}(P_{\psi}) = w^{\ast}( U(\gNS) \circ P_{\psi} ). \]
Since $U(\gNS) \circ P_{\psi} \subset W_{\psi}$ this implies that $U(\gNS) v$ is finite-dimensional
and $\psi$-isotypic. Since $v$ was arbitrary we conclude that $\Hom(M, \Fh_{\psi}(\Hilb_n))$ is $\psi$-isotypic.

The decomposition \eqref{sdas} is unique because \eqref{sum is 1} is unique. Indeed, suppose we had any other decomposition:
\begin{equation} 
\label{other}
\Delta_{\Hilb_n} = \sum_{\psi \in \mathrm{Irrep}(\Fg_{\mathrm{NS}})} P_{\psi}'
\end{equation} 
where $P_{\psi}' \in W_\psi$, for all $\psi$. Then we would need $P_{\psi}' = P_\psi \circ a_\psi$ for some $a_\psi \in W$. But multiplying \eqref{other} on the left with $P_\psi$ and using the orthogonality of the projectors would imply $P_\psi = P_\psi \circ P_\psi \circ a_\psi = P_\psi \circ a_\psi=P_{\psi}'$.  
\end{proof}

As in \cite[Proof of Theorem 7.2]{Mon}, we could also have used Yoneda's Lemma to conclude the existence of the decomposition \eqref{sdas}. Our presentation above has the advantage of being constructive. It also shows that the projectors $P_{\psi}$ can be written in terms of the Nakajima operators
applied to elements in $R^{\ast}(S^k)$.

\section{Multiplicativity}
\label{sec:mult}

\subsection{}  \label{subsec:strategy} 
The main purpose of the present section is to prove Theorems \ref{main} and \ref{thm:ref}.
Let us first discuss the general strategy. Given an operator $H : A^{\ast}(\Hilb_n) \to A^{\ast}(\Hilb_n)$ among $h, h_{\alpha \beta}, h_{\alpha \delta}$, we will first prove a commutation relation of the form:
\begin{equation} [H, \mult_x] = \mult_y \label{Hcomm} \end{equation}
where $\mult_x$ is the operator of multiplication by any $x \in A^{\ast}(\Hilb_n)$
and $y$ will be given by an explicit formula in terms of $x$.
This equation will help us in two ways:
Firstly, applying \eqref{Hcomm} to the fundamental class $1_n \in A^0(\Hilb_n)$ yields:
\begin{equation} H(x) - H(1_n) x = y. \label{Hunit} \end{equation}
Hence if we define $\widetilde{H} = H - H(1_n) \Id_{\Hilb_n}$, then \eqref{Hcomm} reads:
\[ [ \widetilde{H}, \mult_x ] = \mult_{\widetilde{H}(x)}. \]
In other words, we have $\widetilde{H}(x \cdot x') = \widetilde{H}(x) \cdot x' + x \cdot \widetilde{H}(x')$ for all $x,x'$, which precisely states that $\widetilde{H}$ is multiplicative.
Secondly, the explicit formula for $y$ together with~\eqref{Hunit} will yield an expression for $H(x)$, namely $y + H(1_n)x$.
This will be used to determine the value of $H$ on Chern and divisor classes.

\subsection{} In order to prove relations of the form \eqref{Hcomm},
we will introduce the machinery of operators on Chow groups developed by \cite{Lehn, LQW, MN}. The main idea is to develop a common framework for studying the Nakajima operators \eqref{eqn:nakajima} and the following operators:
\begin{equation}
\label{eqn:g}
\fG_d : A^*(\Hilb) \stackrel{\pi^*}\longrightarrow A^*(\Hilb \times S) \xrightarrow{\mult_{\ch_d(\CO_{\CZ})}} A^*(\Hilb \times S)
\end{equation}
(where $\pi : \Hilb \times S^t \rightarrow \Hilb$ denotes the first projection). We will employ the following notation for compositions of these operators, akin to \eqref{eqn:composition 1} and \eqref{eqn:composition 2}:
\begin{equation}
\label{eqn:composition 1 g}
\fG_{d_1}...\fG_{d_t} : A^*(\Hilb) \rightarrow A^*(\Hilb \times S^t)
\end{equation}
where the convention is that $\fG_{d_i}$ acts on the $i$-th factor of $S^t = S \times ... \times S$. Then associated to any $\Gamma \in A^*(S^t)$, one obtains the following endomorphism:
\begin{equation}
\label{eqn:composition 2 g}
\fG_{d_1} ... \fG_{d_t}(\Gamma) = \pi_{*} (\rho^*(\Gamma) \cdot \fG_{d_1} ... \fG_{d_t}) : A^*(\Hilb) \rightarrow A^*(\Hilb)
\end{equation} 
(where $\rho : \Hilb \times S^t \rightarrow S^t$ is the second projection). By a push-pull argument, the expression above is the operator of multiplication by the universal class \eqref{eqn:formula taut 1}:
$$
\fG_{d_1} ... \fG_{d_t}(\Gamma) = \mult_{\univ_{d_1,...,d_t}(\Gamma)}
$$
which explains our interest in the operators \eqref{eqn:g}. 

\subsection{}

Consider the following operators, defined by \cite{LQW} for all $n \in \BZ$ and $d \in \BN \sqcup 0$:
\begin{gather}
\fJ_n^d : A^*(\Hilb) \longrightarrow A^*(\Hilb \times S) \\
\label{eqn:lqw}
\fJ_n^d = d! \left(- \sum_{|\lambda| = n, l(\lambda) = d+1} \frac {\fq_\lambda}{\lambda!}  \Big|_{\Delta} + \sum_{|\lambda| = n, l(\lambda) = d-1} \frac {s(\lambda)+n^2-2}{\lambda!} \rho^*(c) \fq_\lambda \Big|_\Delta \right)
\end{gather}
where for any integer partition $\lambda = (...,(-2)^{m_{-2}},(-1)^{m_{-1}},1^{m_1},2^{m_2},...)$, we define:
\begin{gather*}
l(\lambda) = \sum_{i\in \BZ \backslash 0} m_i, \quad |\lambda| = \sum_{i\in \BZ \backslash 0} i m_i, \quad s(\lambda) = \sum_{i\in \BZ \backslash 0} i^2 m_i,
\quad 
\lambda! = \prod_{i \in \BZ \backslash 0} m_i! \\
\fq_\lambda = ... \fq_2^{m_2} \fq_1^{m_1} \fq_{-1}^{m_{-1}} \fq_{-2}^{m_{-2}} ... : A^*(\Hilb) \rightarrow A^*(\Hilb \times S^{l(\lambda)})
\end{gather*}
and $|_\Delta$ denotes the restriction to the small diagonal $A^*(\Hilb \times S^{l(\lambda)}) \rightarrow A^*(\Hilb \times S)$. For any $\gamma \in A^*(S)$, we may consider the operator:
$$
\fJ_n^d(\gamma) = \pi_*(\rho^*(\gamma) \cdot \fJ_n^d) :A^*(\Hilb) \rightarrow A^*(\Hilb)
$$
and then formula \eqref{eqn:lqw} yields the following:
\begin{multline}\label{eqn:lqw new}
\fJ_n^d(\gamma)  = d! \left(- \sum_{|\lambda| = n, l(\lambda) = d+1} \frac 1{\lambda!}  \cdot \fq_\lambda(\Delta_{1 ... d+1}\gamma_1) \right. \\
\left. + \sum_{|\lambda| = n, l(\lambda) = d-1} \frac {s(\lambda)+n^2-2}{\lambda!} \cdot \fq_\lambda(\Delta_{1 ... d-1}\gamma_1 c_1) \right)
\end{multline}
where $\Delta_{1...d}$ denotes the small diagonal in $S^d$.

\subsection{} 

The following result is proved just like its cohomological counterpart in \cite[Theorem 5.5]{LQW} (the only input the computation needs is relation \eqref{eqn:heis}, which takes the same form in cohomology as in Chow). 

\begin{thm} For all $n,n' \in \BZ$ and $d+d' \geq 3$ (or $d+d' = 2$ but $n+n' \neq 0$):
\begin{equation}
\label{eqn:lqw 0}
[\fJ_n^d, \fJ_{n'}^{d'}] = (dn'-d'n) \Delta_*(\fJ_{n+n'}^{d+d'-1}) + 2 \Omega_{n,n'}^{d,d'} \Delta_*(\rho^*(c) \cdot \fJ_{n+n'}^{d+d'-3}) 
\end{equation}
as operators $A^*(\Hilb) \rightarrow A^*(\Hilb \times S \times S)$, where $\Delta: S \hookrightarrow S \times S$ is the diagonal, $\rho : \Hilb \times S \rightarrow S$ is the second projection, and $\Omega_{n,n'}^{d,d'}$ are certain integers.

\end{thm}

The precise formula for the numbers $\Omega_{n,n'}^{d,d'}$ can be found in relation (5.2) of \cite{LQW} (note that one must replace $n,n' \leftrightarrow -n,-n'$ to match our notation with \emph{loc.~cit.}), but we will only need the following particularly simple cases of formula \eqref{eqn:lqw 0}:
\begin{gather}
[\fq_n, \fJ_0^d] = dn \Delta_*(\fJ_{n}^{d-1}) \label{eqn:lqw 1} \\
[\fL_n, \fJ_0^d] = dn \Delta_*(\fJ_{n}^{d}) + 2 d(d-1)n(n^2-1) \Delta_*(\rho^*(c) \cdot \fJ_{n}^{d-2})  \label{eqn:lqw 2} 
\end{gather}
where we note that $\fJ_n^0 = -\fq_n$, while $\fJ_n^1 = -\fL_n$ with:
\begin{equation}
\label{eqn:virasoro} 
\fL_n = \frac 12 \sum_{i,j \in \BZ, i+j = n} :\!\fq_i \fq_j \Big|_\Delta\!: 
\end{equation}
and $:\! - \!:$ denotes the normal ordered product \eqref{eqn:normal}. 

\begin{thm} [{\cite[Theorem 4.6]{LQW}} in cohomology, {\cite[Theorem 1.7]{MN}} in Chow] For any $d \in \BN$, we have:
$$
\fJ_0^d = d! \left( \fG_{d+1} + 2\rho^*(c) \cdot \fG_{d-1} \right)
$$
as operators $A^*(\Hilb) \rightarrow A^*(\Hilb \times S)$, where $\rho : \Hilb \times S \rightarrow S$ is the second projection. Equivalently, we may write the formula above as:
\begin{equation}
\label{eqn:g to j}
\fJ_0^d(\gamma) = d! \left( \fG_{d+1}(\gamma) + 2 \fG_{d-1} (\gamma c)\right)
\end{equation}
as operators $A^*(\eHilb) \rightarrow A^*(\eHilb)$ indexed by $\gamma \in A^*(S)$. 

\end{thm}

Because $c^2 = 0$, inverting formula \eqref{eqn:g to j} gives:
$$
\fG_d = \frac {\fJ_0^{d-1}}{(d-1)!} - \frac {2 \rho^*(c) \cdot \fJ_0^{d-3}}{(d-3)!}
$$
or equivalently as:
\begin{equation}
\label{eqn:j to g}
\fG_d (\gamma) = \frac {\fJ_0^{d-1}(\gamma)}{(d-1)!} - \frac {2 \fJ_0^{d-3}(\gamma c)}{(d-3)!}.
\end{equation}

\subsection{} Let us now recall the operators: 
\begin{equation}
\label{eqn:operators}
h, h_{\alpha \beta}, h_{\alpha \delta} : A^*(\Hilb) \rightarrow A^*(\Hilb)
\end{equation}
defined in \eqref{eqn:def h}, \eqref{eqn:def h ab}, and \eqref{eqn:def h ad} for all $\alpha, \beta \in A^1(S) \subset A^1(\Hilb)$. In order to prove Theorems \ref{main} and \ref{thm:ref}, we need to compute the commutators of these operators with the operators \eqref{eqn:g} of multiplication by universal classes. 

\begin{prop}
\label{prop:comm 1} 
	
For any $d \geq 2$ and $\alpha, \beta \in A^1(S) \subset A^1(\eHilb)$, we have:
\begin{gather}
\label{eqn:comm h 1 exp}
[h,\fG_d(\gamma)] = \fG_d \left((d-1)\gamma+\int_\bullet \gamma_\bullet(c-c_\bullet) \right) \\
\label{eqn:comm h 2 exp}
[h_{\alpha \beta},\fG_d(\gamma)] = \fG_d \left( \int_\bullet \gamma_{\bullet} (\alpha \beta_\bullet - \alpha_\bullet \beta) \right) 
\end{gather}
as operators $A^*(\eHilb) \rightarrow A^*(\eHilb)$ indexed by $\gamma \in A^*(S)$. 

\end{prop}

Formula \eqref{eqn:comm h 1 exp} is the main thing we need to prove Theorem \ref{main}. However, to prove Theorem \ref{thm:ref}, we will also need the following more complicated version of the formulas above.

\begin{prop}
\label{prop:comm 2}

For any $d \geq 2$ and $\alpha \in A^1(S) \subset A^1(\eHilb)$, we have:
\begin{multline}\label{eqn:comm h 3 exp}
[h_{\alpha \delta},\fG_d(\gamma)] = -\fG_2(\alpha) \fG_{d-1}(\gamma) - \fG_2(1) \fG_{d-1} \left(\gamma \alpha\right) \\
-  \fG_{d+1} \left(\alpha \int_\bullet \gamma_\bullet + \int_\bullet \gamma_\bullet \alpha_\bullet \right) + 2\fG_{d-1} \left(\alpha \int_\bullet \gamma_\bullet c_\bullet \right)
\end{multline}
as operators $A^*(\eHilb) \rightarrow A^*(\eHilb)$ indexed by $\gamma \in A^*(S)$.

\end{prop}

We separate the relations above into two different propositions, because the proof of the latter will be significantly more involved than the former. But before we dive into the proof, let us observe that according to our conventions, we are actually proving~\eqref{eqn:comm h 1 exp},~\eqref{eqn:comm h 2 exp}, and \eqref{eqn:comm h 3 exp} as the following identities:
\begin{gather}
\label{eqn:comm h 1}
(h \times \text{Id}_S) \circ \fG_d - \fG_d \circ h = (d-1) \fG_d + \pi_{2*} \left((c_1-c_2) \cdot \pi_1^*(\fG_d) \right) \\
\label{eqn:comm h 2}
(h_{\alpha \beta} \times \text{Id}_S) \circ \fG_d - \fG_d \circ h_{\alpha \beta} = \pi_{2*} \left((\alpha_1 \beta_2 - \alpha_2 \beta_1) \cdot \pi_1^*(\fG_d) \right)
\end{gather}
and:
\begin{multline}
(h_{\alpha \delta} \times \text{Id}_S) \circ \fG_d - \fG_d \circ h_{\alpha \delta} \\ = \pi_{2*} \Big[ {- (\alpha_1 + \alpha_2)} \cdot \pi_1^*(\fG_2) \pi_2^*(\fG_{d-1}) 
- (\alpha_1+\alpha_2) \cdot \pi_1^*(\fG_{d+1}) 
+ 2 \alpha_1 c_2 \cdot \pi_1^*(\fG_{d-1}) \Big]
\label{eqn:comm h 3}
\end{multline}
of operators $A^*(\Hilb) \rightarrow A^*(\Hilb \times S)$, where $\pi_i: \Hilb \times S \times S \rightarrow \Hilb \times S$ denotes the identity on $\Hilb$ times the projection onto the $i$-th factor of $S$. The reason why we prefer the language of \eqref{eqn:comm h 1 exp}, \eqref{eqn:comm h 2 exp}, and~\eqref{eqn:comm h 3 exp} over \eqref{eqn:comm h 1}, \eqref{eqn:comm h 2}, and \eqref{eqn:comm h 3} is simply to keep the explanation legible.

\subsection{} The reader who is willing to accept Propositions \ref{prop:comm 1} and \ref{prop:comm 2}, and wishes to see how they lead to Theorems \ref{main} and \ref{thm:ref}, may skip to Section \ref{sub:skip}. 

\begin{proof} [Proof of Proposition \ref{prop:comm 1}] We will start with \eqref{eqn:comm h 1 exp}. From a straightforward calculation (see \cite[Lemma 3.4]{Ob}), one obtains the commutation relations:
\begin{equation}
\label{eqn:comm formula 1}
\left[ h, \fq_{\lambda_1} ... \fq_{\lambda_k}(\Phi) \right] = \fq_{\lambda_1} ... \fq_{\lambda_k}(\overline{\Phi})
\end{equation}
for all $\Phi \in A^{\ast}(S^k)$, where we write:
\begin{equation}
\label{eqn:bar 1}
\overline{\Phi} = \sum_{i=1}^k \int_\bullet \underbrace{\Phi_{1 ... i-1,\bullet,i+1 ... k}(c_i - c_\bullet)}_{\text{this class lies in }A^*(S^k \times S)}
\end{equation}
with the last factor in $S^k \times S$ represented by the index $\bullet$.  We have:
\begin{align}
\left[ h, \fJ_0^d(\gamma) \right] & \stackrel{\eqref{eqn:lqw new}}{=} d! \left(- \sum_{|\lambda| = 0, l(\lambda) = d+1} \frac 1{\lambda!}  \cdot [h,\fq_\lambda(\Delta_{1...d+1}\gamma_1)] \right. \\
& \quad\quad{} + \left. \sum_{|\lambda| = 0, l(\lambda) = d-1} \frac {s(\lambda)-2}{\lambda!} \cdot [h,\fq_\lambda(\Delta_{1...d-1}\gamma_1c_1)] \right) \\
& \stackrel{\eqref{eqn:comm formula 1}}{=} d! \left(- \sum_{|\lambda| = 0, l(\lambda) = d+1} \frac 1{\lambda!}  \cdot \fq_\lambda(\overline{\Delta_{1... d+1}\gamma_1}) \right. \label{eqn:temp}\\
& \quad\quad{} + \left. \sum_{|\lambda| = 0, l(\lambda) = d-1} \frac {s(\lambda)-2}{\lambda!} \cdot \fq_\lambda(\overline{\Delta_{1...d-1}\gamma_1c_1}) \right).
\end{align}
To evaluate the expression above, we will need the result below:

\begin{claim}
\label{claim:3}
	
For any $k>0$, $l \geq 0$ and any $\gamma \in A^*(S \times S^l)$, we have:
\begin{equation}
\label{eqn:claim 3}
\overline{\Delta_{1...k} \gamma_1} = \Delta_{1...k} \left[ (k-1) \gamma_{1} + \int_* \gamma_{*} (c_1 - c_*) \right].
\end{equation}
When writing $\gamma_*$, the index $*$ refers to the first factor of $\gamma \in A^*(S \times S^l)$. The other~$l$ factors of $\gamma$ are not involved in the formula above, as the bar notation is defined as in \eqref{eqn:bar 1} with respect to the indices $1,...,k$ only.
	
\end{claim}

\begin{proof} In the sequel, we let $\bullet$ and $*$ denote two different copies of the surface $S$ which will be integrated out. 
Using \eqref{eqn:new identity 3}, the left hand side of \eqref{eqn:claim 3} equals
\begin{multline*}
\sum_{i=1}^{k} \int_\bullet \Delta_{1... i-1,\bullet,i+1 ... k} \gamma_\bullet (c_i - c_\bullet) 
=
\left[ \sum_{i=1}^k \sum_{j \in \{1,...,k,\bullet\} - \{i\}} \int_\bullet \gamma_j \prod_{x \neq i,j} c_x(c_i-c_\bullet) \right] \\
+ \left[ \sum_{i=1}^k \int_\bullet \left(\Delta_{1 ... i-1,\bullet,i+1 ... k} - \sum_{j \in \{1,...,k,\bullet\} - \{i\}} \prod_{x \neq i,j}c_x \right)(c_i-c_\bullet) \int_* \gamma_* c_*  \right] \\
- \left[ \sum_{i=1}^k (k-1) \int_\bullet c_1 ... c_{i-1} c_\bullet c_{i+1} ... c_k (c_i-c_\bullet)\int_* \gamma_* \right].
\end{multline*}
A straightforward calculation using formula \eqref{eqn:new identity 1} shows that the three square brackets above are equal to the corresponding three square brackets below (one needs to integrate out the factor denoted by $\bullet$):
\begin{multline}\label{eqn:aux 1}
\left[ (k-1) \sum_{i = 1}^k \gamma_i \prod_{j \neq i} c_j + k c_1 ... c_k \int_* \gamma_* - \sum_{i=1}^k \prod_{j \neq i} c_j \int_* \gamma_* c_*  \right] \\
+ \left[ \sum_{i=1}^k \left(\Delta_{1... \hat{i} ... k}c_i - (k-1) \prod_{j \neq i} c_j \right) \int_* \gamma_* c_* \right] - \left[ k(k-1) c_1 ... c_k \int_* \gamma_* \right].
\end{multline}
Using the following corollary of \eqref{eqn:bv 6}:
$$
\sum_{i=1}^k \Delta_{1...\hat{i}...k} c_i = (k-2) \Delta_{1...k} + \sum_{i=1}^k c_1 ... c_{i-1} c_{i+1} ... c_k
$$
we may rearrange \eqref{eqn:aux 1} as:
\begin{multline*}
(k-1) \sum_{i = 1}^k \gamma_i \prod_{j \neq i} c_j - \left[ k(k-2) c_1 ... c_k \right] \int_* \gamma_* \\
+ \left[(k-2) \Delta_{1...k} - (k-1) \sum_{i=1}^k \prod_{j \neq i} c_j \right] \int_* \gamma_* c_*.
\end{multline*}
Using formulas \eqref{eqn:new identity 1} and \eqref{eqn:new identity 3}, one recognizes that the expression above equals:
$$
\Delta_{1...k} \left[ (k-1) \gamma_{1} + \int_* \gamma_{*} (c_1 - c_*) \right]
$$
thus establishing formula \eqref{eqn:claim 3}. 
\end{proof}

Relation \eqref{eqn:temp} together with
the following immediate consequence of \eqref{eqn:new identity 1}:
$$
-2 \gamma c+ \int_* \gamma_*  c_*(c-c_*)
= \int_* \gamma_* c (c-c_*)
$$
implies that:
\begin{equation} \label{eqn:temp 2}
[h,\fJ_0^d(\gamma)] = \fJ_0^d \left(d\gamma+\int_* \gamma_*(c-c_*)\right).
\end{equation}
Together with \eqref{eqn:j to g}, formula \eqref{eqn:temp 2} implies \eqref{eqn:comm h 1 exp}.

In order to prove \eqref{eqn:comm h 2 exp}, we will recycle the argument above. By analogy with \eqref{eqn:comm formula 1}, we have:
\begin{equation}
\label{eqn:comm formula 2}
\left[h_{\alpha \beta},\fq_{\lambda_1} ... \fq_{\lambda_k}(\Phi) \right] = \fq_{\lambda_1} ... \fq_{\lambda_k}(\overline{\overline{\Phi}}) 
\end{equation}
for all $\Phi \in A^*(S^k)$, where we write:
\begin{equation}
\label{eqn:bar 2}
\overline{\overline{\Phi}} = \sum_{i=1}^k \int_\bullet \underbrace{\Phi_{1...i-1,\bullet,i+1...k}(\alpha_i \beta_\bullet - \alpha_\bullet \beta_i)}_{\text{this class lies in }A^*(S^k \times S)}
\end{equation}
where the last factor in $S^k \times S$ is the one represented by the index $\bullet$. 

\begin{claim}
\label{claim:4}
	
For any $k>0$, $l \geq 0$ and any $\gamma \in A^*(S \times S^l)$, we have:
\begin{equation}
\label{eqn:claim 4}
\overline{\overline{\Delta_{1 ... k} \gamma_1}} = \Delta_{1...k} \int_* \gamma_{*} (\alpha_1 \beta_* - \alpha_* \beta_1).
\end{equation}
When writing $\gamma_*$, the index $*$ refers to the first factor of $\gamma \in A^*(S \times S^l)$. The other~$l$ factors of $\gamma$ are not involved in the formula above, as the double bar notation is defined as in \eqref{eqn:bar 2} with respect to the indices $1,...,k$ only.
	
\end{claim}

\begin{proof} In the sequel, we let $\bullet$ and $*$ denote two different copies of the surface $S$ which will be integrated out. By definition, the left hand side of \eqref{eqn:claim 4} equals:
\begin{multline*}
\sum_{i=1}^{k} \int_\bullet \Delta_{1... i-1,\bullet,i+1 ... k} \gamma_\bullet (\alpha_i \beta_\bullet - \alpha_\bullet \beta_i) \\
\stackrel{\eqref{eqn:new identity 3}}={} \left[ \sum_{i=1}^k \sum_{j \in \{1,...,k,\bullet\} - \{i\}} \int_\bullet \gamma_j \prod_{x \neq i,j} c_x(\alpha_i \beta_\bullet - \alpha_\bullet \beta_i) \right] \\
+ \left[ \sum_{i=1}^k \int_\bullet \left(\Delta_{1 ... i-1,\bullet,i+1 ... k} - \sum_{j \in \{1,...,k,\bullet\} - \{i\}} \prod_{x \neq i,j}c_x \right)(\alpha_i \beta_\bullet - \alpha_\bullet \beta_i) \int_* c_* \gamma_* \right] \\
- \left[ \sum_{i=1}^k (k-1) \int_\bullet c_1 ... c_{i-1} c_\bullet c_{i+1} ... c_k (\alpha_i \beta_\bullet - \alpha_\bullet \beta_i) \int_* \gamma_* \right].
\end{multline*}
One can apply \eqref{eqn:new identity 2} to compute the square brackets above (integrate out the factor of $S$ denoted by $\bullet$), and obtain:
$$
= \left[ \sum_{i=1}^k \prod_{j \neq i} c_j \left(\alpha_i \int_* \gamma_* \beta_* - \beta_i \int_* \gamma_* \alpha_* \right) \right] + \left[ 0 \right] + \left[ 0 \right].
$$	
Formula \eqref{eqn:bv 4} shows that the right-hand side of the formula above is equal to:
$$
\Delta_{1...k} \int_* \gamma_{*} (\alpha_1 \beta_* - \alpha_* \beta_1)
$$
which establishes \eqref{eqn:claim 4}. 
\end{proof}

The analogue of formula \eqref{eqn:temp} holds with $h$ replaced by $h_{\alpha \beta}$ and the bar replaced by a double bar, hence Claim \ref{claim:4} implies the following analogue of formula \eqref{eqn:temp 2}:
\begin{equation}
\label{eqn:temp 3}
[h_{\alpha \beta}, \fJ_0^d(\gamma)] = \fJ_0^d \left( \int_\bullet \gamma_{\bullet} (\alpha \beta_\bullet - \alpha_\bullet \beta) \right).
\end{equation}
Together with \eqref{eqn:j to g}, this implies \eqref{eqn:comm h 2 exp}. 
\end{proof} 

\subsection{\it Proof of Proposition \ref{prop:comm 2}.} If we recall the definition of the operators of $\fL_k$ in~\eqref{eqn:virasoro}, then formula \eqref{eqn:def h ad} takes the form:
\begin{equation}
\label{eqn:delta}
h_{\alpha \delta} = \sum_{k \neq 0} \frac 1k :\!\fL_k \fq_{-k} (\alpha_1+\alpha_2)\!: \,\,= \sum_{(x,y) \in \{(\alpha,1), (1,\alpha)\}} \sum_{k \neq 0} \frac 1k :\!\fL_k(x) \fq_{-k} (y)\!: \,.
\end{equation}
We may invoke \eqref{eqn:lqw 1} and \eqref{eqn:lqw 2} to obtain:
\begin{multline*}
[h_{\alpha \delta}, \fJ_0^d(\gamma)] = d \sum_{k \neq 0} \sum_{(x,y) \in \{(\alpha,1), (1,\alpha)\}} \Big[ - :\!\fL_k(x) \fJ_{-k}^{d-1}(y \gamma)\!: \\
+ :\!\fJ_{k}^{d}(x \gamma) \fq_{-k}(y)\!: {}+  2(d-1)(k^2-1):\!\fJ_{k}^{d-2} \left( x c\gamma \right) \fq_{-k}(y)\!: \Big].
\end{multline*}
In the normal ordered products above, we put $\fq_k$ and $\fL_k$ to the left of the expression if $k>0$ and to the right of the expression if $k<0$. Since $\fq_0 = 0$ but $\fL_0 = - \fJ_0^1$, we may rewrite the formula above as:  
\begin{multline}\label{eqn:aux 2}
[h_{\alpha \delta}, \fJ_0^d(\gamma)] = -d \fJ_0^1(\alpha) \fJ_0^{d-1}(\gamma) - d \fJ_0^1(1) \fJ_0^{d-1}(\gamma \alpha) \\
+ d \sum_{k \in \BZ} \sum_{(x,y) \in \{(\alpha,1), (1,\alpha)\}} \Big[ - :\!\fL_k(y) \fJ_{-k}^{d-1}(x \gamma)\!: + :\!\fJ_{k}^{d}(x \gamma) \fq_{-k}(y)\!: \\
+ 2(d-1)(k^2-1):\!\fJ_{k}^{d-2} \left( x c\gamma \right) \fq_{-k}(y)\!: \Big].
\end{multline}
Let us compute the formulas on the second and third lines of the formula above. 

\begin{claim}
\label{claim:1}

We have the following formulas:
\begin{multline}\label{eqn:claim 1}
\sum_{k \in \BZ} \sum_{|\lambda| = k, l(\lambda) = d+1} \frac 1{\lambda!} :\! \fq_\lambda (\Delta_{1...d+1} (x\gamma)_1) \fq_{-k}(y)\!: \\
= \sum^{|\mu| = 0}_{l(\mu) = d+2} \frac 1{\mu!} \cdot \fq_\mu \left( \sum_{i=1}^{d+2} \Delta_{1...\hat{i}...d+2} (x\gamma)_{\neq i} y_i \right) 
\end{multline}
and:
\begin{multline}\label{eqn:claim 11}
\sum_{k \in \BZ} \sum_{|\lambda| = k, l(\lambda) = d-1} \frac {s(\lambda) + k^2-2}{\lambda!} :\!\fq_\lambda(\Delta_{1...d-1} (x\gamma c)_1)\fq_{-k}(y)\!: \\
= \sum_{|\mu| = 0, l(\lambda) = d} \frac {s(\mu) -2}{\mu!} \fq_\mu \left( \sum_{i=1}^{d} \Delta_{1...\hat{i}...d} (x\gamma c)_{\neq i} y_i \right) 
\end{multline} 
where $\Delta_{1...\hat{i}...d}x_{\neq i}$ refers to $\Delta_{1...\hat{i}...d} x_j$ for any $j \neq i$.

\end{claim}

\begin{proof} The right-hand side of \eqref{eqn:claim 1} is equal to:
$$
\sum_{\mu = (...,(-2)^{m_{-2}},(-1)^{m_{-1}},1^{m_1},2^{m_2},...)} \frac {... \fq_{-k}^{m_{-k}} ...}{... m_{-k}! ...}  \left(\sum_{i=1}^{l(\mu)} y_i \Delta_{...\hat{i}...}(x\gamma)_{\neq i} \right).
$$
In each summand above, we can pick a copy of $\fq_{-m_k}$ in $m_{-k}$ ways for any $k$, and assign to that copy the insertion $y_i$, and to all other copies the insertion $\Delta_{...\hat{i}...} (x\gamma)_{\neq i}$. The corresponding sum will be term-wise equal to the left-hand side of \eqref{eqn:claim 1}. Formula~\eqref{eqn:claim 11} is proved analogously, so we leave it to the interested reader.
\end{proof} 

\begin{claim}
\label{claim:2}

We have the following formulas:
\begin{multline}\label{eqn:claim 2}
\sum_{k \in \BZ} \sum_{|\lambda| = - k, l(\lambda) = d} \frac 1{\lambda!} :\!\fL_k(y) \fq_\lambda(\Delta_{1...d}(x\gamma)_1)\!: \\
= \sum_{|\mu| = 0, l(\mu) = d+2} \frac 1{\mu!}  \fq_\mu \left(\sum_{1 \leq i < j \leq d+2} \Delta_{1...\hat{i} ... \hat{j} ... d+2} (x\gamma)_{\neq i,j} \Delta_{ij} y_i  \right) \\
- \sum_{|\mu| = 0, l(\mu) = d} \frac {s(\mu)}{2\mu!} \fq_\mu \left(\Delta_{1...d} (xy\gamma)_1 \right)
\end{multline}
and:
\begin{multline}\label{eqn:claim 22}
\sum_{k \in \BZ} \sum_{|\lambda| = - k,l(\lambda) = d-2} \frac {s(\lambda) + k^2-2}{\lambda!} :\!\fL_k(y) \fq_\lambda(\Delta_{1...d-2} (xc\gamma)_1)\!: \\
= \sum_{|\mu| = 0, l(\mu) = d} \frac {s(\mu) -2}{\mu!}  \fq_\mu \left(\sum_{1 \leq i < j \leq d} \Delta_{1...\hat{i} ... \hat{j} ... d} (xc\gamma)_{\neq i,j} \Delta_{ij} y_i  \right) \\
+ \sum_{i,j \in \BZ} \sum_{|\lambda| = - i-j, l(\lambda) = d-2} \frac {ij}{\lambda!} :\!\fq_i \fq_j(\Delta_{12} y_1 ) \fq_\lambda(\Delta_{1...d-2} (xc\gamma)_1)\!: \,.
\end{multline}

\end{claim}

\begin{proof} As in the proof of \eqref{eqn:claim 1}, the terms on the first line of \eqref{eqn:claim 2} are in one-to-one correspondence with the terms on the second line. However, while the latter are normally ordered by definition, the former are not always normally ordered, due to the presence of the following terms:
\begin{gather}
\fq_k \fq_{-l}(\Delta_{12}y_1) \fq_\lambda(\Delta_{1...d}(x\gamma)_1) \label{eqn:norm 1} \\
\fq_\lambda(\Delta_{1...d}(x\gamma)_1) \fq_l \fq_{-k}(\Delta_{12}y_1) \label{eqn:norm 2} 
\end{gather} 
for all $k \geq l \geq 0$ (if $k=l$, the corresponding product appears in both \eqref{eqn:norm 1} and~\eqref{eqn:norm 2}, and we weigh it with weight $1/2$ in both of these formulas). Therefore, the difference between the first and second lines of \eqref{eqn:claim 2} is equal to the work necessary in normally ordering the expressions \eqref{eqn:norm 1} and \eqref{eqn:norm 2}, and we must identify the contribution of these with the expression on the third line of \eqref{eqn:claim 2}. For fixed $k$, this contribution is:
\begin{align*}
- \left(1+ 2 + ... + k-1 + \frac k2 \right) \fq_k \fq_{\lambda'} (\Delta_{1...d} (xy\gamma)_1) & \quad \text{ in the case \eqref{eqn:norm 1}} \\
- \left(1+ 2 + ... + k-1 + \frac k2 \right) \fq_{\lambda''} \fq_{-k} (\Delta_{1...d} (xy\gamma)_1) & \quad \text{ in the case \eqref{eqn:norm 2}}
\end{align*}
where $\lambda'$ (respectively $\lambda''$) denotes $\lambda$ without any one factor $\fq_l$ with $l$ positive (respectively negative). As we sum over all partitions $\lambda$ and over all ways to remove any one factor $\fq_l$ from them, we are left with:
$$
- \sum_{k \in \BZ} \sum_{|\lambda|=-k, l(\lambda) = d-1} \frac {k^2}{2 \lambda!} :\!\fq_k \fq_\lambda(\Delta_{1...d} (xy\gamma)_1)\!:
$$
which is precisely the third line of \eqref{eqn:claim 2}. Formula \eqref{eqn:claim 22} is proved analogously, only that we do not have to worry about the commutators that arose in the preceding paragraph, because $xyc\gamma = 0$ for any $(x,y) \in \{(\alpha,1), (1,\alpha)\}$. The expression on the last line of \eqref{eqn:claim 22} simply arises as the difference $k^2-i^2-j^2 = s(\lambda) + k^2 - s(\lambda \sqcup \{i,j\})$ in the notation thereof.
\end{proof}

For every $k \in \BN$ consider the cycles in $S^k$ defined by:
\begin{align*}
A_k(\gamma) & = \sum_{(x,y) \in \{(\alpha,1),(1,\alpha)\}} \sum_{1 \leq i < j \leq k} \Delta_{1...\hat{i} ... \hat{j} ... k} (x\gamma)_{\neq i,j}  \Delta_{ij} y_i  \\
B_k(\gamma) & = \sum_{(x,y) \in \{(\alpha,1),(1,\alpha)\}}  \sum_{i=1}^{k} \Delta_{1...\hat{i}...k} (x\gamma)_{\neq i} y_i .
\end{align*} 
Then using Claims \ref{claim:1} and \ref{claim:2} and \eqref{eqn:lqw new}, we may rewrite formula \eqref{eqn:aux 2} as:
\begin{multline}\label{eqn:aux 3}
[h_{\alpha \delta}, \fJ_0^d(\gamma)] = -d \fJ_0^1(\alpha) \fJ_0^{d-1}(\gamma) - d \fJ_0^1(1) \fJ_0^{d-1}(\gamma \alpha) 
+ d! \left[ \sum_{l(\mu) = d+2} \frac 1{\mu!} \fq_\mu \left( A_{d+2}(\gamma) \right) \right. \\
- \sum_{(x,y)} \sum_{l(\mu) = d} \frac {s(\mu)}{2\mu!} \fq_\mu \left(\Delta_{1...d} (xy\gamma)_1 \right) 
- \sum_{l(\mu) = d} \frac {s(\mu) -2}{\mu!}  \fq_\mu \left( A_{d}(c \gamma) \right) \\
- \sum_{(x,y)} \sum_{i,j \in \BZ} \sum_{|\lambda| = - i - j, l(\lambda) = d-2} \frac {ij}{\lambda!} :\!\fq_i \fq_j(\Delta_{12} y_1 ) \fq_\lambda(\Delta_{1...d-2} (xc\gamma)_1)\!: \\
- d\sum_{l(\mu) = d+2} \frac 1{\mu!} \fq_\mu \left( B_{d+2}(\gamma) \right) 
+ d\sum_{l(\mu) = d} \frac {s(\mu) -2}{\mu!} \fq_\mu \left( B_d(c \gamma) \right) \\
- \left. \sum_{(x,y)} \sum_{k \in \BZ} \sum_{|\lambda| = k, l(\lambda) = d-1} \frac {2(k^2-1)}{\lambda!} :\!\fq_\lambda (\Delta_{1...d-1} (xc\gamma)_1) \fq_{-k}(y)\!: \right]
\end{multline}
where above and hereafter, all the partitions denoted by $\mu$ will have $|\mu| = 0$
and $(x,y)$ runs over $\{(\alpha,1), (1,\alpha)\}$.

\begin{claim}
\label{claim:5}

The sum of the third and fifth lines of \eqref{eqn:aux 3} equals: 
\begin{equation}
\label{eqn:claim 5}
2 \sum_{l(\mu) = d} \frac 1{\mu!} \fq_\mu \left( \sum_{i=1}^{d} \Delta_{1...\hat{i}...d} (c\gamma)_{\neq i} \alpha_i  \right).
\end{equation}

\end{claim}

\begin{proof} Because $\alpha c = 0$, only the $(x,y) = (1,\alpha)$ has a non-zero contribution to the third and fifth lines of \eqref{eqn:aux 3}, which means that their sum equals (using \eqref{eqn:bv 3} and~\eqref{eqn:bv 4}):
\begin{multline}\label{eqn:aux 4}
- \sum_{k \in \BZ} \sum_{|\lambda| = k, l(\lambda) = d-1} \frac {2(k^2-1)}{\lambda!} :\! \fq_\lambda (\gamma_1 c_1 ... c_{d-1}) \fq_{-k}(\alpha)\!: \\
- \sum_{i,j \in \BZ} \sum_{|\lambda| = - i - j, l(\lambda) = d-2} \frac {ij}{\lambda!} :\!\fq_i \fq_j(\alpha_1c_2+c_1\alpha_2) \fq_\lambda(\gamma_1c_1 ... c_{d-2})\!: \,.
\end{multline}
Because $\alpha c = 0$, all the $\fq_i$'s commute in the formula above, hence the second line of \eqref{eqn:aux 4} equals:
$$
-2\sum_{i\in \BZ} :\! i \fq_i(\alpha) \underbrace{\sum_{j \in \BZ} \sum_{|\lambda| = - i - j, l(\lambda) = d-2} \frac {j}{\lambda!} \fq_j(c) \fq_\lambda(\gamma_1c_1 ... c_{d-2})}\!:\,. 
$$
The underbraced sum is equal to:
$$
\sum_{|\lambda| = -i, l(\lambda) = d-1} \frac {|\lambda|}{\lambda!} \fq_\lambda (\gamma_1 c_1 ... c_{d-1}).
$$
Plugging this fact into \eqref{eqn:aux 4} leads to:
$$
2 \sum_{k \in \BZ} \sum_{|\lambda| = k, l(\lambda) = d-1} \frac {1}{\lambda!} :\! \fq_\lambda (\gamma_1 c_1 ... c_{d-1}) \fq_{-k}(\alpha)\!:
$$
which is equal to \eqref{eqn:claim 5} by a straightforward rearranging of terms (akin to the one we performed in Claim \ref{claim:1}).
\end{proof}

Using Claim \ref{claim:5}, and after reordering terms, we may rewrite \eqref{eqn:aux 3} as:
\begin{multline}\label{eqn:aux 5}
[h_{\alpha \delta}, \fJ_0^d(\gamma)] = -d \fJ_0^1(\alpha) \fJ_0^{d-1}(\gamma) -d \fJ_0^1(1) \fJ_0^{d-1}(\gamma \alpha) \\
 + d!\left[ \sum_{l(\mu) = d+2} \frac 1{\mu!}  \fq_\mu \left( A_{d+2}(\gamma) - d B_{d+2} (\gamma) \right) + 2 \sum_{l(\mu) = d} \frac 1{\mu!}  \fq_\mu \left(A_d(\gamma c) - (d-1) B_d(\gamma c) \right)\right. \\
\left. + \sum_{l(\mu) = d} \frac {s(\mu)}{\mu!} \fq_\mu \left( d B_d(\gamma c) - A_d(\gamma c)  - \Delta_{1...d} \left( \alpha_1 \int_\bullet \gamma_\bullet c_\bullet + c_1 \int_\bullet \gamma_\bullet \alpha_\bullet \right) \right)  \right]
\end{multline}
where in the last term, we used \eqref{eqn:new identity 2}.

\begin{lemma} \label{ABrelations} We have:
\begin{gather}
\label{eqn:aux 6} 
A_k(\gamma) - (k-2) B_k(\gamma) = \Delta_{1...k} \left(\alpha_1 \int_\bullet \gamma_\bullet + \int_\bullet \alpha_\bullet \gamma_\bullet \right) \\
\label{eqn:aux 7}
A_k(\gamma c) = (k-1) \Delta_{1...k} \alpha_1 \int_\bullet \gamma_\bullet c_\bullet,   \quad B_k(\gamma c) = \Delta_{1...k} \alpha_1 \int_\bullet \gamma_\bullet c_\bullet.
\end{gather}
\end{lemma}
\begin{proof}
The equations in the second line follow immediately from \eqref{eqn:bv 3}.
The first line follows from \eqref{eqn:bv 6} and the following claim.
\end{proof}
\begin{claim}
\label{eqn:claim 6}
For any $\alpha \in A^1(S)$ and any $\gamma \in A^*(S \times S^l)$, we have the following:
\begin{multline}\label{eqn:ak}
A_k(\gamma) = (k-2) \sum_{i \neq j} \alpha_i \gamma_j \prod_{s \neq i,j} c_s \\
- (k-1)(k-3) \left( \sum_{i} \alpha_i \prod_{j \neq i} c_j \right) \left( \int_\bullet \gamma_\bullet \right) + \left (\sum_{i < j} \Delta_{ij} \prod_{s \neq i,j} c_s  \right) \left( \int_\bullet \gamma_\bullet \alpha_\bullet \right) \\ 
+ (k-2) \left( \sum_{i \neq s < t \neq i} \alpha_i \Delta_{st} \prod_{j \neq i,s,t} c_j 
- (k-3) \alpha_i \sum_{j \neq i} \prod_{s \neq i,j} c_s \right)  \left( \int_\bullet \gamma_\bullet c_\bullet \right)
\end{multline}
and:
\begin{multline}\label{eqn:bk}
B_k(\gamma) = \sum_{i \neq j} \alpha_i \gamma_j \prod_{s \neq i,j} c_s \\
- (k-2) \left( \sum_{i} \alpha_i \prod_{j \neq i} c_j \right) \left( \int_\bullet \gamma_\bullet \right) + \left (\sum_{i} \prod_{j \neq i} c_j  \right) \left( \int_\bullet  \gamma_\bullet \alpha_\bullet  \right) \\
+ \left( \sum_{i \neq s < t \neq i} \alpha_i \Delta_{st} \prod_{j \neq i,s,t} c_j 
- (k-3) \alpha_i \sum_{j \neq i} \prod_{s \neq i,j} c_s \right)  \left( \int_\bullet \gamma_\bullet c_\bullet \right).
\end{multline}
\end{claim}

\begin{proof} Let us prove \eqref{eqn:bk} and leave the analogous formula \eqref{eqn:ak} as an exercise to the interested reader. Formulas \eqref{eqn:bv 6} and \eqref{eqn:new identity 3} imply:
\begin{multline*}
\Delta_{...\hat{i}...}(x\gamma)_{\neq i} y_i 
= \sum_{i \neq j} \gamma_j x_j y_i \prod_{s \neq i,j} c_s \\
+ \left(\sum_{i \neq s < t \neq i} \Delta_{st} y_i \prod_{j \neq i,s,t} c_j 
- (k-2) y_i \sum_{j \neq i} \prod_{s \neq i,j}c_s \right) \left( \int_\bullet \gamma_\bullet x_\bullet c_\bullet \right)\\
- (k-2) y_i \prod_{j\neq i} c_j \left(\int_\bullet \gamma_\bullet  x_\bullet\right).
\end{multline*}
If we sum over $i \in \{1,...,k\}$ and over $(x,y) \in \{(\alpha, 1), (1,\alpha)\}$, we obtain \eqref{eqn:bk} (note that in the $(x,y) = (\alpha, 1)$ case of the first term in the right-hand side, we need to use formula \eqref{eqn:new identity 2} to calculate $\gamma_j \alpha_j$). 
\end{proof}

With Lemma~\ref{ABrelations} in mind, \eqref{eqn:aux 5} reads:
\begin{multline*}
[h_{\alpha \delta}, \fJ_0^d(\gamma)] = -d \fJ_0^1(\alpha) \fJ_0^{d-1}(\gamma) -d \fJ_0^1(1) \fJ_0^{d-1}(\gamma \alpha) \\
+ d!\left[ \sum_{l(\mu) = d+2} \frac 1{\mu!}  \fq_\mu \left( \Delta_{1...d+2} \left(\alpha_1 \int_\bullet \gamma_\bullet + \int_\bullet \gamma_\bullet \alpha_\bullet \right) \right) \right. \\
\left. - \sum_{l(\mu) = d} \frac {s(\mu)}{\mu!} \fq_\mu \left(\Delta_{1...d} c_1 \int_\bullet \gamma_\bullet \alpha_\bullet  \right)\right].
\end{multline*}
By \eqref{eqn:lqw}, the second and third lines of the expression above equal:
$$
d!\left[- \frac 1{(d+1)!} \fJ_0^{d+1} \left( \alpha \int_\bullet \gamma_\bullet + \int_\bullet \gamma_\bullet \alpha_\bullet  \right) + \frac 2{(d-1)!} \fJ_0^{d-1} \left( c \int_\bullet \gamma_\bullet \alpha_\bullet \right)\right].
$$
If we convert the $\fJ$'s to $\fG$'s in the formula above using \eqref{eqn:g to j}, we obtain \eqref{eqn:comm h 3 exp}. \qed

\subsection{}
\label{sub:skip} 

Before we prove Theorems \ref{main} and \ref{thm:ref}, let us compute how the operators \eqref{eqn:operators} act on the fundamental class.

\begin{lem} 
\label{lem:tiny}
	
If $1_n \in A^*(\Hilb_n)$ denotes the fundamental class, then $h(1_n) = -n$ and $h_{\alpha \beta}(1_n) = h_{\alpha \delta} (1_n) = 0$ for all $\alpha, \beta \in A^1(S) \subset A^1(X)$.
	
\end{lem} 

\begin{proof} It is well-known that:
$$
1_n = \frac 1{n!} \fq_1(1)^n (1_0).
$$
Because the only operator $\fq_k$ which fails to commute with $\fq_1$ is $\fq_{-1}$, formula \eqref{eqn:def h} implies that:
\begin{align*}
h(1_n) &= h \left(\frac 1{n!} \fq_1(1)^n (1_0) \right) = \left[h, \frac 1{n!} \fq_1(1)^n \right] (1_0) \\
&= \left[\fq_1(1)\fq_{-1}(c) - \fq_1(c) \fq_{-1}(1), \frac 1{n!} \fq_1(1)^n \right] (1_0) \\
&= \sum_{i=1}^n \frac 1{n!} \fq_1(1)^{i-1}  \cdot \fq_1(1) [\fq_{-1}(c) , \fq_{1}(1)] \cdot \fq_{-1}(1)^{n-i}\cdot (1_0) = - n 1_n
\end{align*}
where the fact that $[\fq_{-1}(c), \fq_1(1)] = -1$ is a consequence of \eqref{eqn:heis op}. The fact that $[\fq_{-1}(\alpha), \fq_1(1)] = 0$ for any $\alpha \in A^1(S)$ means that the analogous computation implies that $h_{\alpha \beta}(1_n) = 0$. Similarly, let us use formula \eqref{eqn:delta} to compute:
\begin{align}
h_{a\delta}(1_n) &= \frac 1{n!} \left[h_{a \delta}, \fq_1(1)^n\right](1_0) \\
&= \frac 2{n!} \sum_{k = 1}^\infty \frac 1k \left[\fL_k \fq_{-k} (\alpha_1+\alpha_2) - \fq_{k} \fL_{-k} (\alpha_1+\alpha_2), \fq_1(1)^n\right] (1_0).\label{eqn:temp 6}
\end{align}
As a consequence of \eqref{eqn:heis op}, $\fq_{-k}(\alpha)$ and $\fq_{-k}(1)$ commute with $\fq_1(1)$. Meanwhile, the well-known Virasoro algebra relation (\cite{Lehn}) reads $[\fL_k(\gamma), \fq_1(1)] = - \fq_{k+1}(\gamma)$ for all~$k \in \BZ$ and $\gamma \in A^*(S)$. Therefore, all the commutators vanish in \eqref{eqn:temp 6} (or more precisely, they all have an annihilating operator $\fq_{-k}$ with $k \geq 0$ on the very right, and therefore act by 0 on $1_0$) hence $h_{a\delta}(1_n) = 0$.
\end{proof} 

\subsection{} By iterating relation \eqref{eqn:comm h 1 exp} $t$ times, we infer the following formula for all $d_1,...,d_t \geq 2$ and all $\Gamma \in A^*(S^t)$:
$$
[h,\fG_{d_1} ... \fG_{d_t}(\Gamma)] = \fG_{d_1} ... \fG_{d_t} \left(\sum_{i=1}^t (d_i-1) \Gamma + \int_\bullet \sum_{i=1}^t \Gamma_{1...i-1, \bullet, i+1...t} (c_i-c_\bullet) \right)
$$
(note that we are actually using formula \eqref{eqn:comm h 1} in order to conclude the aforementioned result). Since $\fG_{d_1}...\fG_{d_t}(\Gamma)$ is the operator of multiplication by $\univ_{d_1,...,d_t}(\Gamma)$, we conclude that:
\begin{equation}
\label{eqn:temp 5}
[h,\mult_{\univ_{d_1,...,d_t}(\Gamma)}] = \mult_{\univ_{d_1,...,d_t}(\Gamma')}
\end{equation}
where $\Gamma' = \sum_{i=1}^t (d_i - 1) \Gamma + \int_\bullet \sum_{i=1}^t \Gamma_{1...i-1, \bullet, i+1...t} (c_i-c_\bullet)$.
We are now ready to prove Theorem~\ref{main}, and follow the strategy outlined in Section~\ref{subsec:strategy}.

\begin{proof}[Proof of Theorem \ref{main}] Let us first prove part (i). By applying \eqref{eqn:temp 5} to the fundamental class $1_n \in A^*(\Hilb_n)$, we obtain:
$$
h(\univ_{d_1,...,d_t}(\Gamma)) - \univ_{d_1,...,d_t}(\Gamma) \cdot h(1_n) = \univ_{d_1,...,d_t}(\Gamma').
$$	
As a consequence of Lemma \ref{lem:tiny}, we obtain:
\begin{equation}
\label{eqn:temp 7}
\univ_{d_1,...,d_t}(\Gamma') = h(\univ_{d_1,...,d_t}(\Gamma)) +n \cdot \univ_{d_1,...,d_t}(\Gamma) = \th(\univ_{d_1,...,d_t}(\Gamma))
\end{equation}
on $A^*(\Hilb_n)$. Therefore, relation \eqref{eqn:temp 5} reads:
\begin{equation}
\label{eqn:temp 8}
[\th, \mult_{\univ_{d_1,...,d_t}(\Gamma)}] = \mult_{\th(\univ_{d_1,...,d_t}(\Gamma))}.
\end{equation}
As a consequence of the surjectivity of the morphism \eqref{eqn:surj}, we conclude that:
\begin{equation} \label{eqn:temp 8.5}
[\th, \mult_x] = \mult_{\th(x)}
\end{equation}
as an equality of operators $A^*(\Hilb_n) \rightarrow A^*(\Hilb_n)$, indexed by any $x \in A^*(\Hilb_n)$. This is precisely equivalent to \eqref{eqn:derivation}. In the language of correspondences, relation~\eqref{eqn:temp 8} is viewed as an equality of correspondences in $A^*(\Hilb_n \times \Hilb_n \times S^t)$. By \eqref{eqn:last} and~\eqref{eqn:surj}, there exist correspondences:
\[Z \in A^* (\Hilb_n \times \bigsqcup_a S^t) \quad \text{ and }\quad W \in A^*(\bigsqcup_a S^t \times \Hilb_n)\]
(here the indexing set follows the notation at the end of Section \ref{sub:ooo}) such that $W \circ Z = \Delta_{\Hilb_n}$. By composing the third factor of~\eqref{eqn:temp 8} with $^tZ \in A^*(\sqcup_a S^t \times \Hilb_n)$, we obtain~\eqref{eqn:temp 8.5} as an equality of correspondences in~$A^*(\Hilb_n \times \Hilb_n \times \Hilb_n)$, which is what our result claims.

Let us now prove part (ii), which requires us to show that~$\th(x) = \deg(x) \cdot x$ if $x$ is a divisor class or a Chern class of the tangent bundle. According to Section \ref{sub:chern} it suffices to show this relation for:
\begin{equation}
\label{eqn:temp 9}
x = \univ_{d_1,...,d_t} (\Gamma) 
\end{equation}
where $t = 1$ and $\Gamma = \gamma \in \{1,l,c\}_{l \in A^1(S)}$, or $t = 2$ and $\Gamma = \Delta_*(\gamma) \in A^*(S^2)$ for $\gamma \in \{1,c\}$. The degree of such a class $x$ is:
$$
\deg x = d_1+...+d_t + \deg \Gamma -2t.
$$
As a consequence of \eqref{eqn:temp 7}, we have: 
$$
\th(x) = \univ_{d_1,...,d_t} \left( (d_1+...+d_t-t)\Gamma + \int_\bullet\sum_{i=1}^t \Gamma_{1...i-1, \bullet, i+1...t} (c_i-c_\bullet) \right)
$$
so the class $x$ lies in the appropriate direct summand if:
\begin{equation}
\label{eqn:want}
\int_\bullet \sum_{i=1}^t \Gamma_{1...i-1, \bullet, i+1...t} (c_i-c_\bullet) = (\deg \Gamma -t) \Gamma.
\end{equation}
If $t=1$ and $\Gamma = \gamma \in \{1,l,c\}_{l \in A^1(S)}$, this relation is trivial, while if $t = 2$ and $\Gamma = \Delta_*(\gamma)$ for $\gamma \in \{1,c\}$, it is an immediate consequence of Claim \ref{claim:3}. 
\end{proof}

\begin{proof} [Proof of Theorem \ref{thm:ref}] The proof follows that of Theorem \ref{main} very closely. 
For part~(i) we iterate relations \eqref{eqn:comm h 2} and \eqref{eqn:comm h 3} to obtain:
\begin{align*}
[h_{\alpha \beta},\mult_{\univ_{d_1,...,d_t}(\Gamma)}] &= \mult_{\univ_{d_1,...,d_t}(\Gamma')} \\
[h_{\alpha \delta},\mult_{\univ_{d_1,...,d_t}(\Gamma)}] &= \mult_{\sum_{d_1',...,d_t'} \univ_{d'_1,...,d'_t}(\Gamma'')}
\end{align*}
where in the right-hand sides, $\Gamma',\Gamma''$ are obtained from $\Gamma \in A^*(S^t)$ by pulling back to some $S^{t+t'}$, multiplying with certain cycles, and pushing forward to $S^t$ again. In either case, we may apply the relations above to the fundamental class (and invoke Lemma \ref{lem:tiny}) to conclude that:
\begin{align*}
h_{\alpha \beta} \left( \univ_{d_1,...,d_t}(\Gamma) \right) &= \univ_{d_1,...,d_t}(\Gamma') \\
h_{\alpha \delta} \left( \univ_{d_1,...,d_t}(\Gamma) \right) &= \sum_{d_1',...,d_t'} \univ_{d'_1,...,d'_t}(\Gamma'').
\end{align*} 
Therefore, we conclude that:
\begin{align*} 
[h_{\alpha \beta},\mult_{\univ_{d_1,...,d_t}(\Gamma)}] &= \mult_{h_{\alpha \beta}(\univ_{d_1,...,d_t}(\Gamma))} \\
[h_{\alpha \delta},\mult_{\univ_{d_1,...,d_t}(\Gamma)}] &= \mult_{h_{\alpha \delta}(\univ_{d_1,...,d_t}(\Gamma))}.
\end{align*} 
As explained at the end of the proof of Theorem \ref{main} (i), the formulas above imply~\eqref{eqn:derivation ref} as an equality of correspondences in $A^*(\Hilb_n \times \Hilb_n \times \Hilb_n)$.

For part (ii), one can check the identity
$h_{\alpha \beta}(c_k(\Tan_{\Hilb_n})) = 0$
in the same way as the analogous proof in Theorem \ref{main} (ii). Hence it remains to prove that:
\[ h_{\alpha \delta}(\ch_k(\Tan_{\Hilb_n})) = 0. \]
By Lemma~\ref{lem:tiny}, it suffices to show that the commutator
of $h_{\alpha \delta}$ with the operator of multiplication by $\ch_k(\Tan_{\Hilb_n})$ vanishes.
By Section~\ref{sub:chern} this operator is precisely:
\begin{multline} \label{mult_tan}
\mult_{\ch_k(\Tan_{\Hilb_n})} = 2 \fG_{k+2}(1) + 4 \fG_k(c) \\
+ \sum_{i+j=k+2} (-1)^{j+1} \fG_i \fG_j(\Delta) + 2 \sum_{i+j=k} (-1)^{j+1} \fG_i(c) \fG_j(c).
\end{multline}
Using \eqref{eqn:comm h 3 exp} one has for all $d$ the commutation relations:
\begin{align*}
[h_{\alpha \delta}, \fG_{d}(1) ] & = - \fG_2(\alpha) \fG_{d-1}(1) - \fG_2(1) \fG_{d-1}(\alpha) + 2 \fG_{d-1}(\alpha) \\
[h_{\alpha \delta}, \fG_{d}(c) ] & = - \fG_2(\alpha) \fG_{d-1}(c) - \fG_{d+1}(\alpha)
\end{align*}
and for all $i,j \geq 2$ the relations:
\begin{align*}
[h_{\alpha \delta}, \fG_{i} \fG_{j}(\Delta) ] & = 
{-\fG_{2}(\alpha)} (\fG_{i-1} \fG_{j} + \fG_{i} \fG_{j-1})(\Delta) \\
& \quad\quad -\fG_{2}(1) (\fG_{i-1} \fG_{j} + \fG_{i} \fG_{j-1})(\Delta \alpha_1) \\
& \quad\quad - (\fG_{i+1} \fG_{j} + \fG_{i} \fG_{j+1})(\alpha_1 + \alpha_2) \\
& \quad\quad + 2 \fG_{i-1}(\alpha) \fG_{j}(c) + 2 \fG_{i}(c) \fG_{j-1}(\alpha).
\end{align*}
Using the above identities, it is straightforward to show that $h_{\alpha \delta}$ commutes with the right-hand side of \eqref{mult_tan} (we leave the computation as an exercise to the interested reader). This implies the required equation, namely $[h_{\alpha \delta}, \mult_{\ch_k(\Tan_{\Hilb_n})}] = 0$.
\end{proof}

\end{document}